\documentclass{article}
\usepackage{maa-monthly}
\usepackage{subfigure}
\theoremstyle{theorem}
\newtheorem{theorem}{Theorem}

\newtheorem{corollary}[theorem]{Corollary}
\newtheorem{example}{Example}

\theoremstyle{definition}
\newtheorem{definition}[theorem]{Definition}

\begin{document}

\title{Two Binary trees of Rational numbers——the S-tree and the SC-tree}
\markright{The S-tree and the SC-tree}
\author{Ziting Wang \and Ruijia Guo \and Yixin Zhu }

\maketitle

\begin{abstract}
In this study, we explore a novel approach to demonstrate the countability of rational numbers and illustrate the relationship between the Calkin-Wilf tree and the Stern-Brocot tree in a more intuitive manner. By employing a growth pattern akin to that of the Calkin-Wilf tree, we construct the S-tree and establish a one-to-one correspondence between the vertices of the S-tree and the rational numbers in the interval $(0,1]$ using 0-1 sequences. To broaden the scope of this concept, we further develop the SC-tree, which is proven to encompass all positive rational numbers, with each rational number appearing only once. We also delve into the interplay among these four trees and offer some applications for the newly introduced tree structures.

\end{abstract}


\section{Introduction}
\ 

The set of rational numbers is demonstrably countable, a fact exemplified by the conventional approach employing a serpentine enumeration on a square lattice, as depicted in Figure \ref{countable}. Nonetheless, it is worth acknowledging that within this method, numerous redundancies emerge, for instance, $\frac{1}{2}=\frac{2}{4}=\frac{3}{6}=\frac{4}{8}=\cdots$. These repetitive representations may obfuscate the method's inherent intuitiveness.
\begin{figure}[ht]
    \centering
    \includegraphics[width=0.55\textwidth]{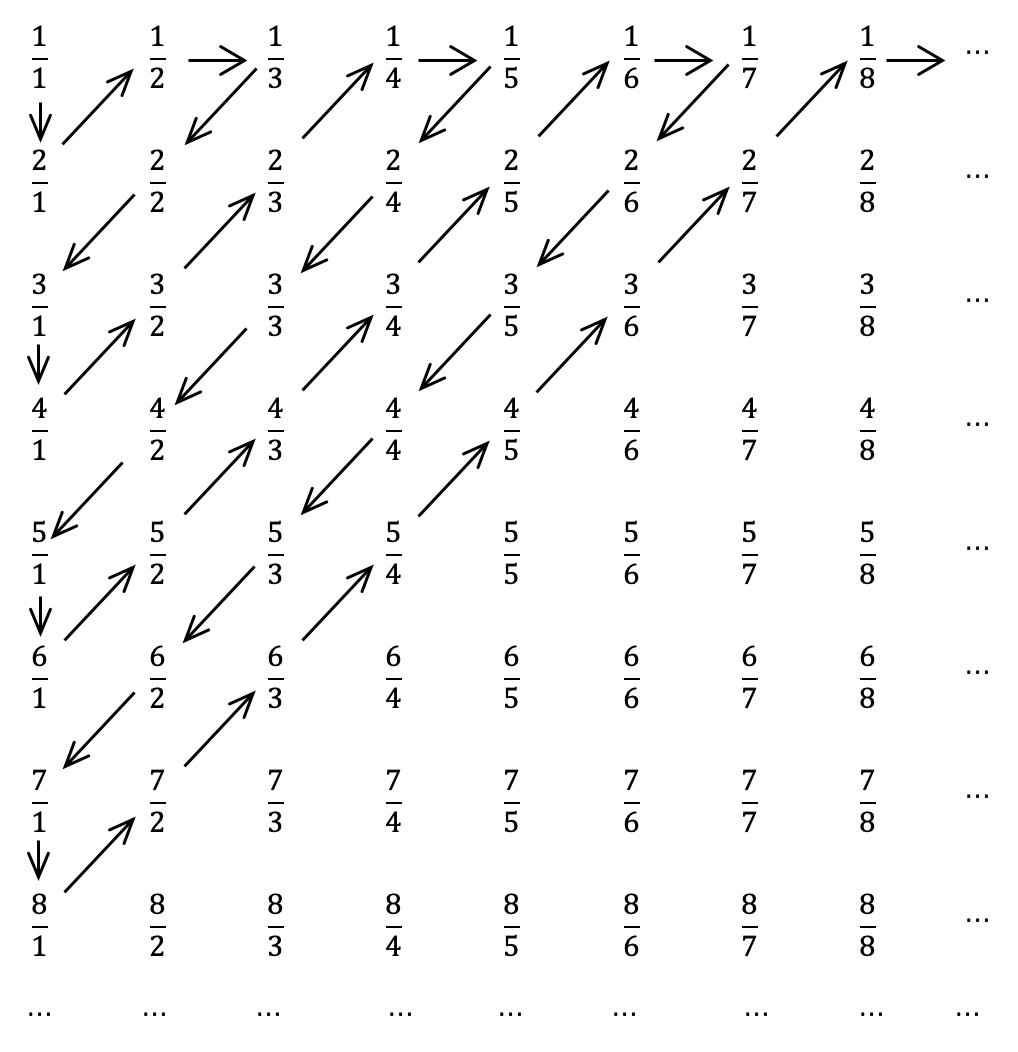}
    \caption{Illustration of the countability of the positive rationals}
    \label{countable}
\end{figure} 

It is of significance to highlight that both the Stern-Brocot tree\cite{Stern, Brocot} and the Calkin-Wilf tree\cite{Calkin} are two renowned binary tree structures that offer a one-to-one correspondence with reduced positive rational numbers, eliminating any such redundancy.

The Stern-Brocot tree is an infinite binary tree, rooted in the Stern-Brocot sequence, a sequence independently discovered by Moritz Stern (1858) and Achille Brocot (1861)\cite{Stern, Brocot}. To generate new terms at level $n+1$ in the Stern-Brocot sequence from two adjacent terms at level $n$, they employed the Mediants method, represented as $\frac{a_1}{b_1}\bigoplus \frac{a_2}{b_2}=\frac{a_1+a_2}{b_1+b_2}$, as visually illustrated in Figure \ref{Stern-Brocot-image}.

\begin{figure}[ht]
    \centering
    \includegraphics[width=0.6\textwidth]{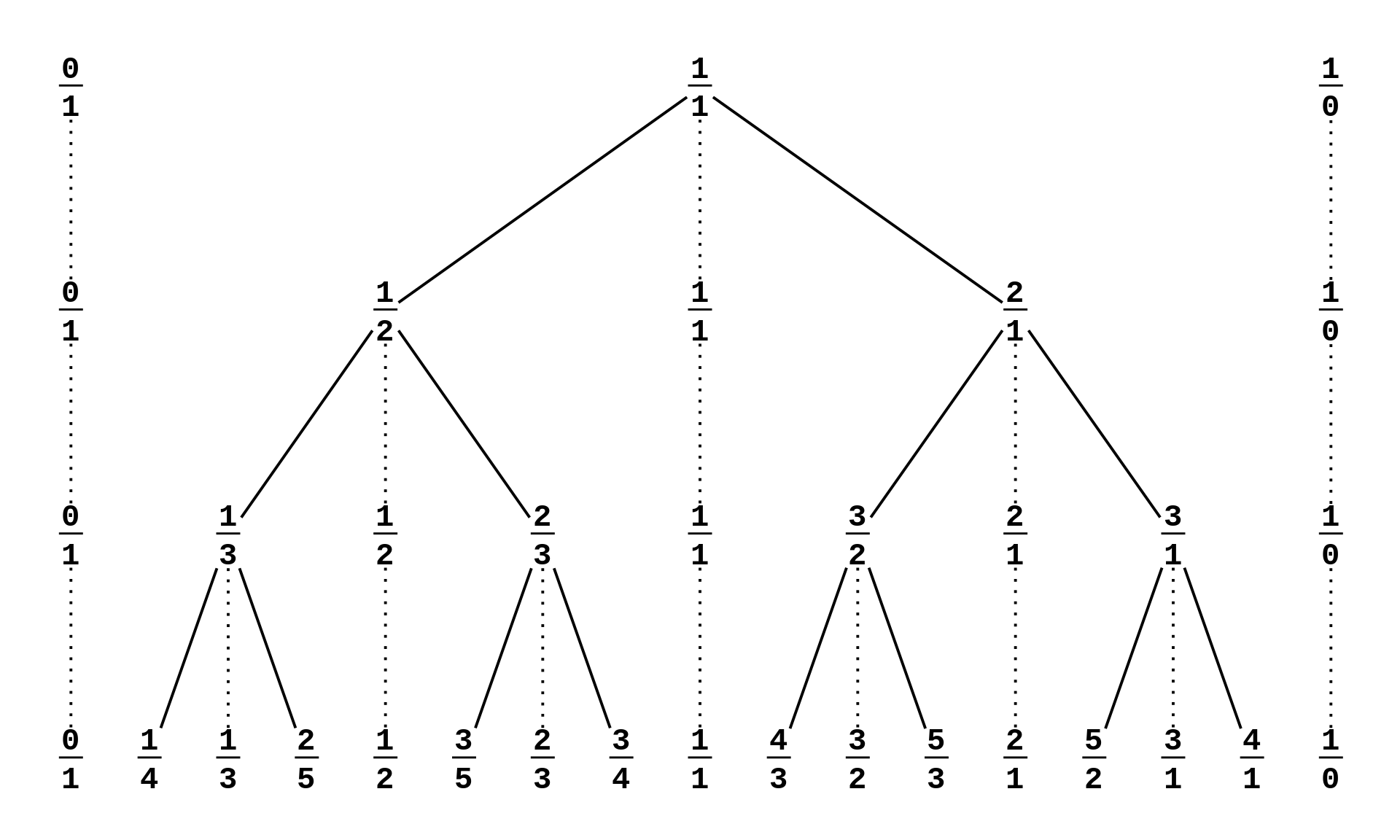}
    \caption{The Stern–Brocot tree, and the Stern–Brocot sequences of level 1-4 }
    \label{Stern-Brocot-image}
\end{figure}

In contrast, the construction of the Calkin-Wilf tree follows a more straightforward approach. The Calkin-Wilf tree takes $\frac{1}{1}$ as its root vertex, and for any vertex $\frac{a}{b}$, its left child is $\frac{a}{a+b}$, and its right child is $\frac{a+b}{b}$, as depicted in Figure \ref{Calkin-Wilf-image}.

\begin{figure}[ht]
    \centering
    \includegraphics[width=0.5\textwidth]{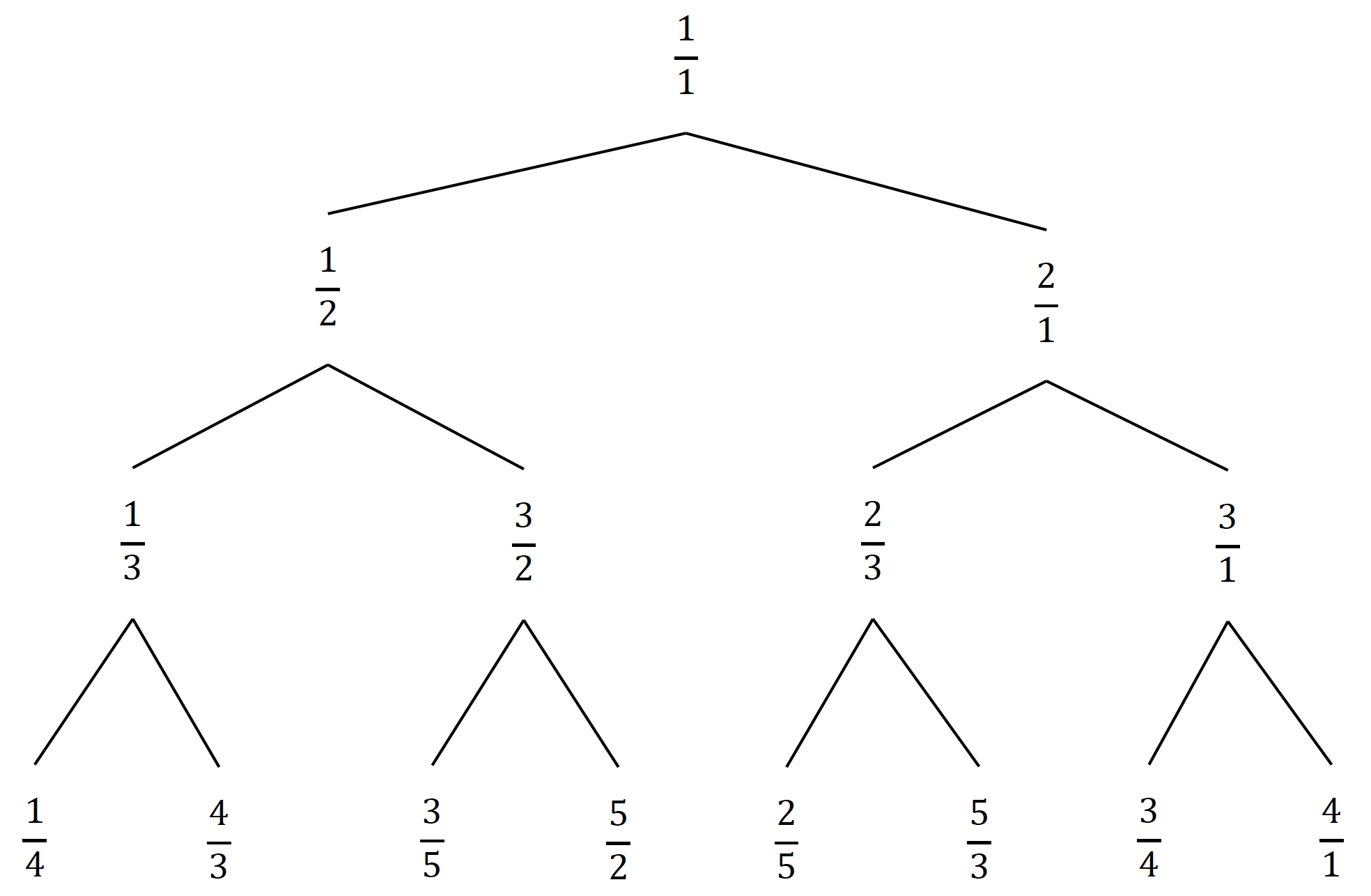}
    \caption{The Calkin-Wilf tree of leverl 1-4}
    \label{Calkin-Wilf-image}
\end{figure}

Discussions concerning these binary trees predominantly revolve around their interrelation and the algorithms designed to determine the position of a given rational number or its inverse counterpart, namely, discovering the corresponding rational number for a given position. Among these results, the algorithms highlighting the connections between the Stern-Brocot tree and the Calkin-Wilf tree might be the most noticeable. Nevertheless, existing findings still lack a certain degree of intuitiveness. Bruce Bates et al. employed the continued fractions of rationals and additional factors to address the challenges of vertex positioning in both the Stern-Brocot tree and the Calkin-Wilf tree\cite{Bates2010, Bates}.

The addition factors they defined as the role of indicators of $j$ just like: $$\lambda_i={\lceil \log_2{j_i}\rceil}\ ,\ i=0,1,2,\cdots ,k$$ where $$j_i=\left\{
\begin{aligned}
    &j,\ \ \ \ \ \ \ \ \ \ \ \ \ \ \ \ \ \ \ \ \ \mathrm{for}\ \ i=0 \\
    &2^{\lambda_{i-1}}-j_{i-1},\ \ \mathrm{for}\ \ i=1,2,\cdots,k \\
\end{aligned}
    \right.$$ 
and $k$ is the smallest value of $i$ for which $\log_2{j_i}= \lceil \log_2{j_i}\rceil $, $j_k =2^{\lambda_{k}}$ ( $\lceil x\rceil$ represents the ceiling function, which rounds the real number 
$x$ to the smallest integer not less than 
$x$ )\cite{Bates}. 
More skilfully, Joao F. Ferreira and Alexandra Mendes delved into the realm of matrices. They discovered that by pre-multiplying $(1,1)^{T}$ or post-multiplying $(1,1)$ to each matrix on the tree of the product of \textbf{L} and \textbf{R} (as depicted in Figure \ref{LR-image}), where the matrices \textbf{L} and \textbf{R} are defined as
$\textbf{L}=\begin{pmatrix}
    1 & 0\\
    1 & 1
\end{pmatrix}$ and $\textbf{R}=\begin{pmatrix}
    1 & 1\\
    0 & 1
\end{pmatrix}$ .  
The vertices of the Stern-Brocot tree and the Calkin-Wilf tree can be determined as follows: for the vector $(x,y)^{T}$, it corresponds to the rational number $\frac{x}{y}$ on the Stern-Brocot tree, and for the vector $(x,y)$, it corresponds to the rational number $\frac{y}{x}$ on the Calkin-Wilf tree\cite{Ferreira}. However, it is worth noting that both the addition factors and the matrix forms can be seen as somewhat specialized techniques. The addition factors, to some extent, lack comprehensive details regarding their connections and the final algorithms, especially concerning their relationship with continued fractions. As for the matrices, their calculations may demand practiced skills, potentially introducing challenges when attempting to locate the position of a given rational number. Consequently, these approaches are more akin to verifiable theorems, accurate but somewhat devoid of heuristic intricacies.

\begin{figure}[ht]
    \centering
    \includegraphics[width=0.85\textwidth]{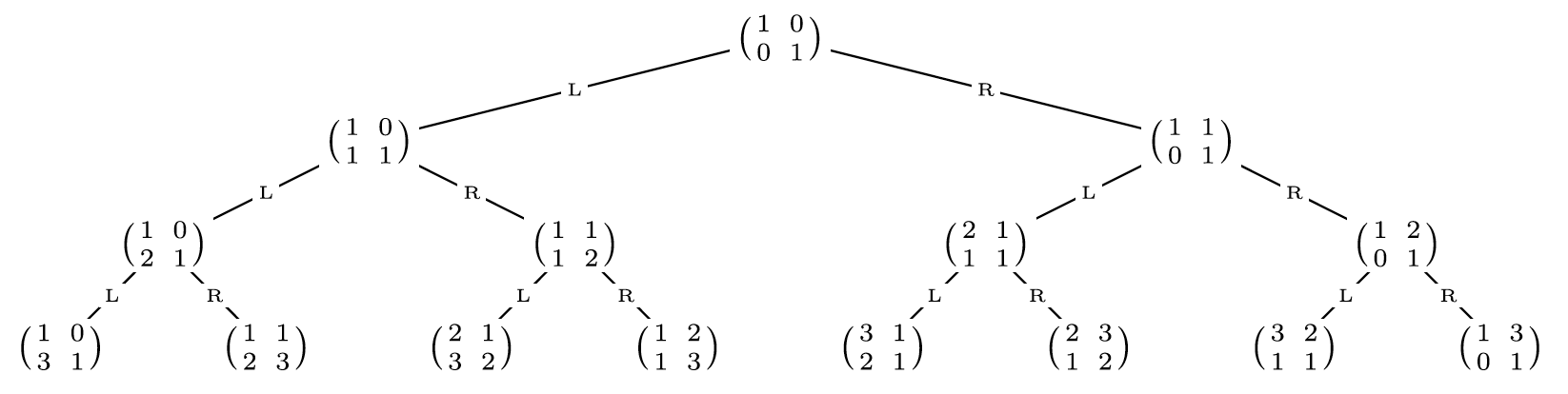}
    \caption{The tree of product of \textbf{L} and \textbf{R}}
    \label{LR-image}
\end{figure}

To enhance clarity, it would be intriguing to explore the links between the process of generating vertices and the construction of continued fractions. In an effort to provide a more intuitive perspective on both the countability of rational numbers and the structures of the two binary trees, this paper introduces the S-tree and the SC-tree. These constructs aim to strike a balance between intuitiveness and complexity while capturing the fundamental patterns governing the positioning of vertices within the trees. We seek to establish a relationship between these new trees and their representation through 0-1 sequences.

To elaborate further, we intend to demonstrate that both the S-tree and the SC-tree offer an intuitive explanation for the countability of rational numbers. Moreover, by mapping 0 to the left and 1 to the right, this paper delves into the direct correlation between the structure of continued fractions for a given rational number and their corresponding 0-1 sequence within the aforementioned tree. Through this endeavor, a more cohesive and easily comprehensible link can be established among the four aforementioned trees. Consequently, this approach opens the door to various potential applications.

The structure of this paper is as follows. In Section 2, we provide the definitions of the S-tree and the SC-tree before demonstrating that they are correspondent one-to-one to the reduced rationals in $(0,1]$ and $(0,+\infty)$, respectively. In addition, we aim to present the algorithms for locating their vertices. In Section 3, we demonstrate the complete linking loop of the four trees by establishing the correspondence between the SC-tree and the Stern-Brocot tree, along with the partial linkage of the S-tree and the Calkin-Wilf tree. In Section 4, we introduce some applications that can be addressed by using the S-tree. In Section 5, we provide concluding remarks of the paper.

\section{The S-tree and the SC-tree}

\ 

As highlighted in the Introduction, several studies present crucial theorems and descriptions concerning the relationship between the Stern-Brocot tree and the Calkin-Wilf tree. Nevertheless, there is a need to enhance the intuitiveness of these concepts. Therefore, in this section, we will introduce two structures. These structures not only offer an intuitive proof of the countability of rational numbers but also provide a concise perspective on the connection between the Stern-Brocot tree and the Calkin-Wilf tree.

\subsection{The S-tree and its properties }
\ 

In this subsction, we define the S-tree and then prove that every positive rational in the interval $(0,1]$ occurs exactly once on the S-tree. 

    \begin{figure}[ht]
    \centering
    \includegraphics[width=0.19\textwidth]{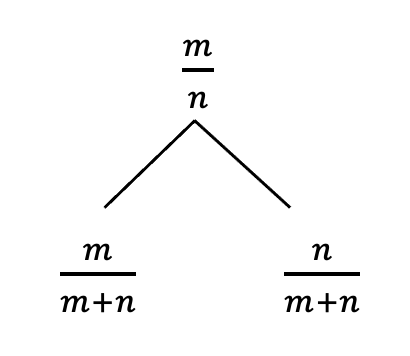}
    \caption{Growing pattern of the S-tree}
    \label{S-min-e1}
\end{figure}

\begin{definition}\label{def:e1}[S-tree]
The S-tree is an infinite complete binary tree where fractions grow on the tree according to the following two rules:
\begin{itemize}
    \item The top of the tree constitutes level 0, where $\frac{1}{1}$ is located.
    \item  For every vertex $\frac{m}{n}$ in level $k$ ( $k\ge 1$), it has its left child $\frac{m}{m+n}$ and its right child $\frac{n}{m+n}$ in level $k+1$. ( $\frac{1}{1}$ has only one child, $\frac{1}{2}$, except this, all other vertices have two different children).
\end{itemize}
\end{definition}
\noindent The growing pattern shown by Figure \ref{S-min-e1}, and Figure \ref{S-tree-e1} shows levels 0-3 of the S-tree. 

\begin{figure}[ht]
    \centering
    \includegraphics[width=0.65\textwidth]{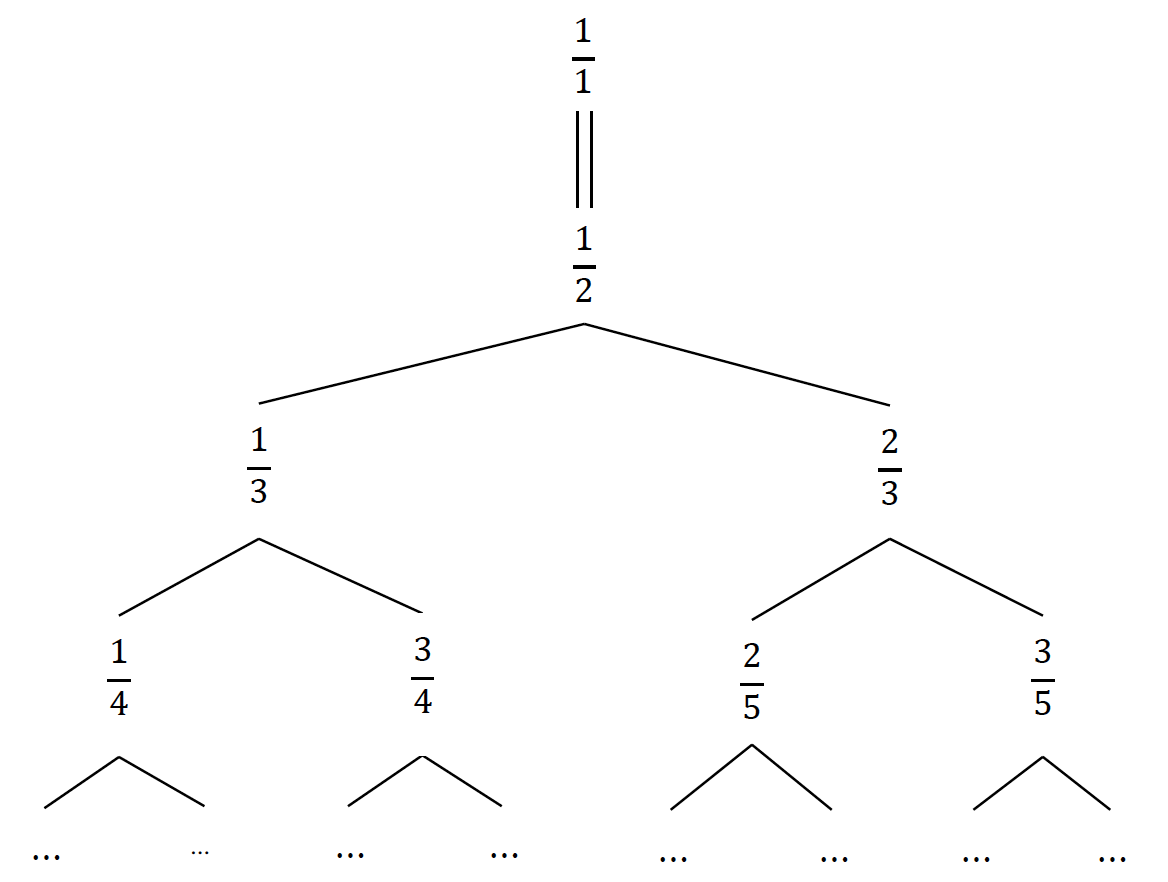}
    \caption{Levels 0-3 of the S-tree}
    \label{S-tree-e1}
\end{figure}

For all positive integers $a$ and $b$, it is evident that $a + b$ is greater than both $a$ and $b$, and $a + b$ belongs to the set of positive integers. Additionally, the greatest common factors of $(a, a + b)$ and $(b, a + b)$ are both equal to $(a, b)$. These properties can be readily observed from the S-tree:

\begin{itemize}
    \item Except $\frac{1}{1}$, every vertex's right child is greater than its left child, and the sum of these two children remains a constant 1. 
    \item All vertices are reduced rational numbers in the interval $(0,1].$
\end{itemize}
    
\begin{theorem}\label{thm:e1}
The S-tree corresponds one-to-one to reduced rational numbers in the interval $(0,1]$. 
\end{theorem}

\begin{proof}
\ \ 

\begin{itemize}
    \item \textit{Every reduced positive rational number in the interval $(0,1]$ occurs at some vertices of the S-tree.} 
\end{itemize}
\noindent $\frac{1}{1}$ certainly occurs.  Otherwise, let $\frac{u}{v} \in (0,1),(u,v)=1$ be, among all fractions that do not occur, one of smallest denominator, and among the one of smallest numerator. If $v-u>u>0$, then $\frac{u}{v-u}$ doesn't occur, else its left child would be $\frac{u}{v}$, but the denominator of $\frac{u}{v-u}$ is smaller, a contradiction. If $0<v-u<u$, then $\frac{v-u}{u}$ would similarly results a contradiction. 
\begin{itemize}
    \item \textit{No reduced positive rational number in the interval $(0,1]$ occurs at more than one vertex of the S-tree.} 
\end{itemize}
\noindent $\frac{1}{1}$ certainly occurs once of the tree(if not, it would be a child of some vertices $\frac{u}{v} \in (0,1)$, but neither $\frac{u}{v+u}$ nor $\frac{v}{v+u}$ can be equal to $\frac{1}{1}$). Otherwise, let $\frac{u}{v} \in (0,1)$ with $(u,v)=1$ be, among all fractions that occur more than once, one of smallest denominator, and among the one of smallest numerator. If $v-u>u>0$, then $\frac{u}{v-u}$ occurs more than once, otherwise, its left child $\frac{u}{v}$ would only occur once, which would lead to a contradiction due to the smaller denominator of $\frac{u}{v-u}$. Similarly if $0<v-u<u$. 
\end{proof}

When counting the vertices starting from the top root and proceeding in a top-down, left-to-right order, we can easily observe that Theorem \ref{thm:e1} reveals the existence of an injective function from rationals in $(0,1]$ to the natural numbers. 

In this subsection, we built the S-tree and demonstrated that it can establish a one-to-one correspondence with reduced fractions in the interval $(0,1]$, as shown by Theorem \ref{thm:e1}.
\subsection{Locating vertices of the S-tree}
\ 

Based on Definition \ref{def:e1} and Theorem \ref{thm:e1}, every vertex can be literally located on the finite path from the top vertex to it, and subtraction is the direct operation involved when seeking its predecessors. This implies that it would be computationally easy. In this subsection, we demonstrate that, in practice, even if a fraction has a large numerator or denominator, the task of locating it is feasible when using binary arithmetic. We introduce 0-1 sequences as indicators of positions on a binary tree before linking them to the corresponding rationals.

\begin{definition}\label{def:e2}[0-1 sequence of a binary tree]
All vertices from level 2 of a binary tree can be ordered with 0-1 sequences according to the following rules:

\begin{itemize}
    \item Starting from the level that contains exactly two vertices (excluding the level with only one vertex), the left vertex is assigned the $\overline{0}$, and the right vertex is assigned the $\overline{1}$.
    \item  To generate the 0-1 sequence of a vertex's left child, add '$0$' to the right side of the vertex's 0-1 sequence. Adding '$1$' generates the 0-1 sequence of the right child.
\end{itemize}
\end{definition}

\noindent For instance, vertices $\frac{1}{4}$, $\frac{3}{4}$, $\frac{2}{5}$, $\frac{3}{5}$ of the S-tree can be assigned the 0-1 sequences $\overline{00}$, $\overline{01}$, $\overline{10}$, $\overline{11}$ respectively\footnote{Note that $\frac{1}{1}$ and $\frac{1}{2}$, the only two vertices of the S-tree without corresponding 0-1 sequences defined above, are located at the top of the tree, and their positions are obvious, so they do not affect the locating task. }.  Actually, when converting a 0-1 sequence into its binary expression, the value can indicate the serial number of the vertex in its level. For example, the vertex $\frac{2}{5}$ of the S-tree is ordered with $\overline{10}$, the value of the binary expression $\overline{10}$ is $2$, so $\frac{2}{5}$ is the 3rd vertex in its level ( counted from the far left ). The 0-1 sequence defined by Definition \ref{def:e2} is a way to describe any binary tree since every vertex on it only has two nodes as its children. This concept will be frequently referred to in the content below.

Now we introduce continued fraction  expansion as a preparatory step for the positioning algorithm. The notation $[a_0,a_1,\cdots,a_n]$ (set $a_n\geq 2$) is used to represent the continued fraction expansion$$a_0+\frac{1}{a_1+\frac{1}{\ddots +\frac{1}{a_{n-1}+\frac{1}{a_n}}}}.$$

\begin{theorem}\label{thm:en2}
For a vertex $q\in (0,1)$ of the S-tree, given by the continued fraction $q=[0,a_1,\cdots,a_k]$, it occupies the position of the $N$-th vertex in level $M$, where $M=-1+\Sigma_{i=1}^k a_i$, and $N=1+\Sigma_{i=1}^{k-1}2^{(-1+\Sigma_{j=1}^{i}a_j)}$ ( when $k=1$, set $N=1$). 
\end{theorem}

\begin{proof}
\ \ 

\noindent Expressing $\frac{m}{n}$ as $\frac{1}{1+\frac{n-m}{m}}$ or $\frac{1}{1+\frac{1}{\frac{m}{n-m}}}$, allows us to understand that, if $[0,a_1,a_2,\cdots,a_n]$ represents $\frac{n-m}{m}$, then $\frac{m}{n}$ must be $[0,1,a_1,a_2,\cdots,a_n]$. Furthermore, if  $[0,a_1,a_2,\cdots,a_n]$ equals $\frac{m}{n-m}$, then $\frac{m}{n}$ must be $[0,a_1+1,a_2,\cdots,a_n]$.In the context of the S-tree,  this means that adding '$1$' as the second digit of the continued fraction of a vertex can generate its right child, while adding '$1$' to the second digit can generate its left child. 

\noindent The fraction $\frac{1}{2}$, which has a continued fraction expansion $[0,2]$, obviously satisfies Theorem \ref{thm:en2}. Except for $\frac{1}{2}$ itself, all fractions in the interval $(0, 1)$ are descended from it. It's worth noting that $\frac{1}{a_k}=[0,a_k]$ is the first vertex in level $a_k-1$, so its 0-1 sequence is $\overline{\underbrace{0\cdots0}_{a_k-2}}$. Therefore, we can trace the continued fraction $q=[0,a_1,\cdots,a_k]$ starting from $[0,a_k]$. 

\noindent By induction on $k$, we can see the following result: If $a_{k-1}=1$, then $[0,a_{k-1},a_k]$ is the right child of $[0,a_k]$, so the 0-1 sequence of $[0,a_{k-1},a_k]$ is $\overline{\underbrace{0\cdots0}_{a_k-2}1}$. If $a_{k-1}> 1$, then the 0-1 sequence of $[0,a_{k-1},a_k]$ must be $\overline{\underbrace{0\cdots0}_{a_k-2}1\underbrace{0\cdots0}_{a_{k-1}-1}}$. Continuing this process, we can determine that the 0-1 sequence of $q=[0,a_1,\cdots,a_k]$ is $\overline{\underbrace{0\cdots0}_{a_k-2}1\underbrace{0\cdots0}_{a_{k-1}-1}1\underbrace{0\cdots0}_{a_{k-2}-1}\cdots 1\underbrace{0\cdots0}_{a_{1}-1}}$. The length of this 0-1 sequence is $\Sigma_{i=1}^k a_i-2$. Referring to Definition \ref{def:e2}, we can determine that the level where $q$ is located is $M=-1+\Sigma_{i=1}^k a_i$. Furthermore, it is the $(1+\Sigma_{i=1}^{k-1}2^{(-1+\Sigma_{j=1}^{i}a_j)})$-th vertex in that level (when $k=1$, set $N=1$).

\noindent The result follows.
\end{proof}

\begin{example}\label{exam:01-en}
   The continued fraction expansion of $\frac{328}{853}$ is $\frac{1}{2+\frac{1}{1 +\frac{1}{1+\frac{1}{1 +\frac{1}{1+\frac{1}{65}}}}}}$, denoted as $\frac{328}{853}=[0,2,1,1,1,1,65].$ From Theorem \ref{thm:en2}, we have $$M=-1+2+1+1+1+1+65=70,$$
\begin{equation*}
    \begin{split} N&=1+2^{-1+2+1+1+1+1}+2^{-1+2+1+1+1}+2^{-1+2+1+1}+2^{-1+2+1}+2^{-1+2}\\
    &=1+2^5+2^4+2^3+2^2+2^1\\
    &=63.
    \end{split}
\end{equation*}

\noindent Therefore, $\frac{328}{853}$ is the 63rd vertex in level 70 of the S-tree.
\end{example}

The unit fraction $\frac{1}{n}$ is always the first vertex in level $n-1$ of the S-tree, so $\frac{n-1}{n}$ is consistently the second vertex in level $n-1$. Similarly, for the S-tree, we have:

\begin{corollary}\label{coro:01-en}
\ \ 

\begin{itemize}
    \item $\forall n>2,n\in Z_{+}$, the vertex $\frac{n-1}{n}$ corresponds 0-1 sequence $\overline{\underbrace{0\cdots0}_{n-3}1}$. 
    \item $\forall n>1,n\in Z_{+}$, the vertex $\frac{2n-1}{2n+1}$ corresponds 0-1 sequence $\overline{1\underbrace{0\cdots0}_{n-2}1}$. 
    \item $\forall n>1,n\in Z_{+}$, the vertex $\frac{3n-2}{3n+1}$ corresponds 0-1 sequence $\overline{01\underbrace{0\cdots0}_{n-2}1}$, and the vertex $\frac{3n-1}{3n+2}$ corresponds $0-1$ sequence $\overline{11\underbrace{0\cdots0}_{n-2}1}$. 

\end{itemize}
\end{corollary}

\noindent Of course, Corollary\ref{coro:01-en} can be expanded more broadly to any reduced rational number $\frac{n-d}{n}, d,n\in Z_{+}, 0<d<n$, and the number of the 0-1 sequence categories would be $\varphi(d)$, where $\varphi(x)$ represents the Euler function.

As shown by Theorem \ref{thm:e1}, all vertices of the S-tree can be one-to-one corresponded with reduced positive rationals in the interval $(0,1]$. Moreover, Theorem \ref{thm:en2} naturally associates the 0-1 sequence and continued fraction of a vertex, and it provides an algorithm for locating the position of a vertex of the S-tree. While it is easy to extend the property that rationals within $(0,1]$ are countable to the set $Q_+$, we choose to construct a tree similar to the Calkin-Wilf tree or the Stern-Brocot tree, which encompasses all positive rationals. In the next subsection, we will delve into the SC-tree , which serves as an optimal solution for the expanding purpose. 

\subsection{The SC-tree and its properties }
\ 

In this subsection, we will introduce a new structure called the SC-tree, designed to extend the range of corresponding rationals beyond the constraints of the S-tree, which is limited to the interval $(0,1]$. To achieve this, we will outline its properties and demonstrate that every vertex of the SC-tree can establish a one-to-one correspondence with rationals from the set $\mathbf{Q_+}$.

When tracking the origin of every vertex's denominator on the S-tree, we can gain deeper insight into another structure within it. Replace the root $\frac{1}{1}$ with $\frac{a}{b}$ and keep the growth rules, then the S-tree would appear as the tree shown in Figure \ref{image:a-b-tree-en}, where every vertex has a denominator in the form of '$ma+nb$'.

 \begin{figure}[ht]
  \centering
  \includegraphics[width=0.95\textwidth]{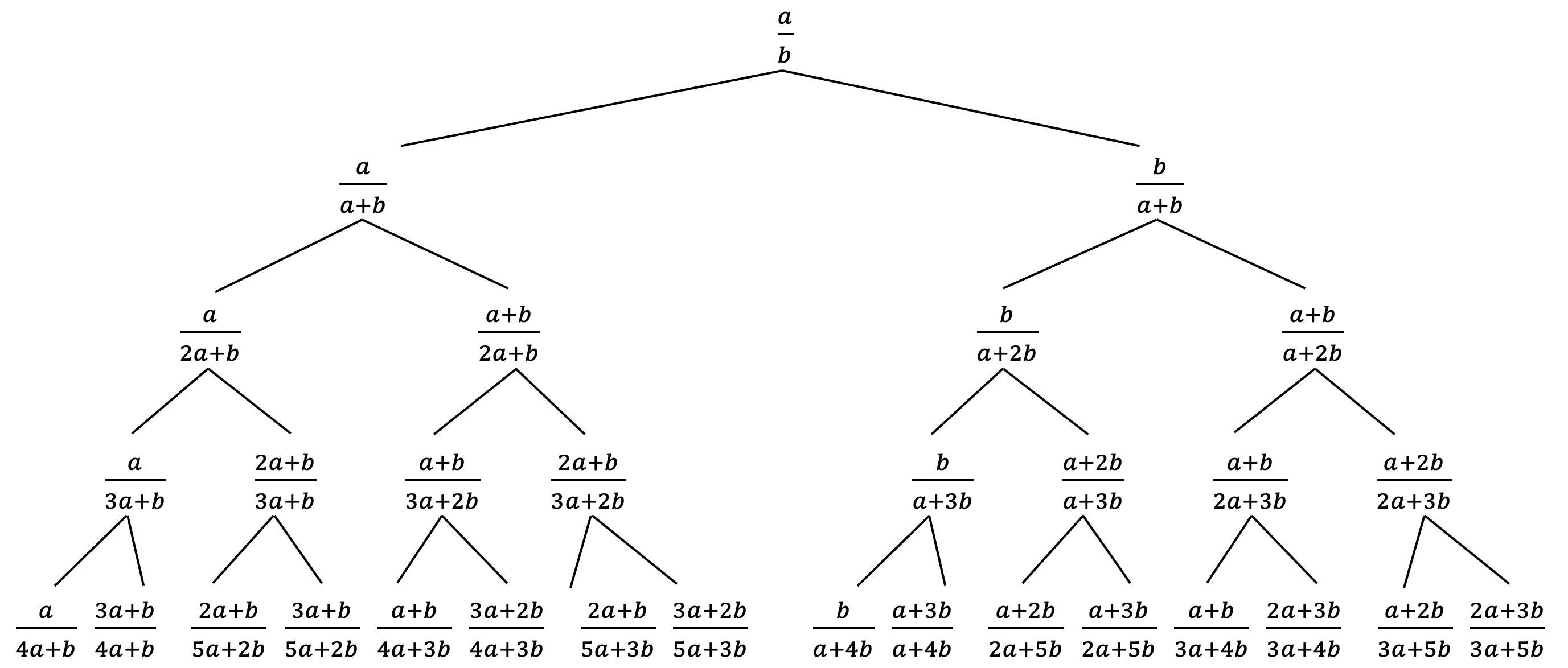}
  \caption{\label{image:a-b-tree-en}S-tree with the root vertex $\frac{a}{b}$}
\end{figure}    

\begin{definition}\label{def:e3}[SC-tree]
The SC-tree is an infinite complete binary tree defined by the following two rules:
\begin{itemize}
    \item  Replace every two vertices of the S-tree whose denominators are both '$ma+nb$' with one fraction $\frac{m}{n}$.
    \item  Add two pseudo-fractions, $\frac{1}{0}$ and $\frac{0}{1}$, at level 0. Define the fraction $\frac{1}{1}$ as the right child of $\frac{1}{0}$ and the left child of $\frac{0}{1}$.
\end{itemize}
\end{definition}

\noindent Figure \ref{image:Cab-tree-en} shows levels 0-4 of the SC-tree.  

 \begin{figure}[ht]
  \centering
  \includegraphics[width=0.86\textwidth]{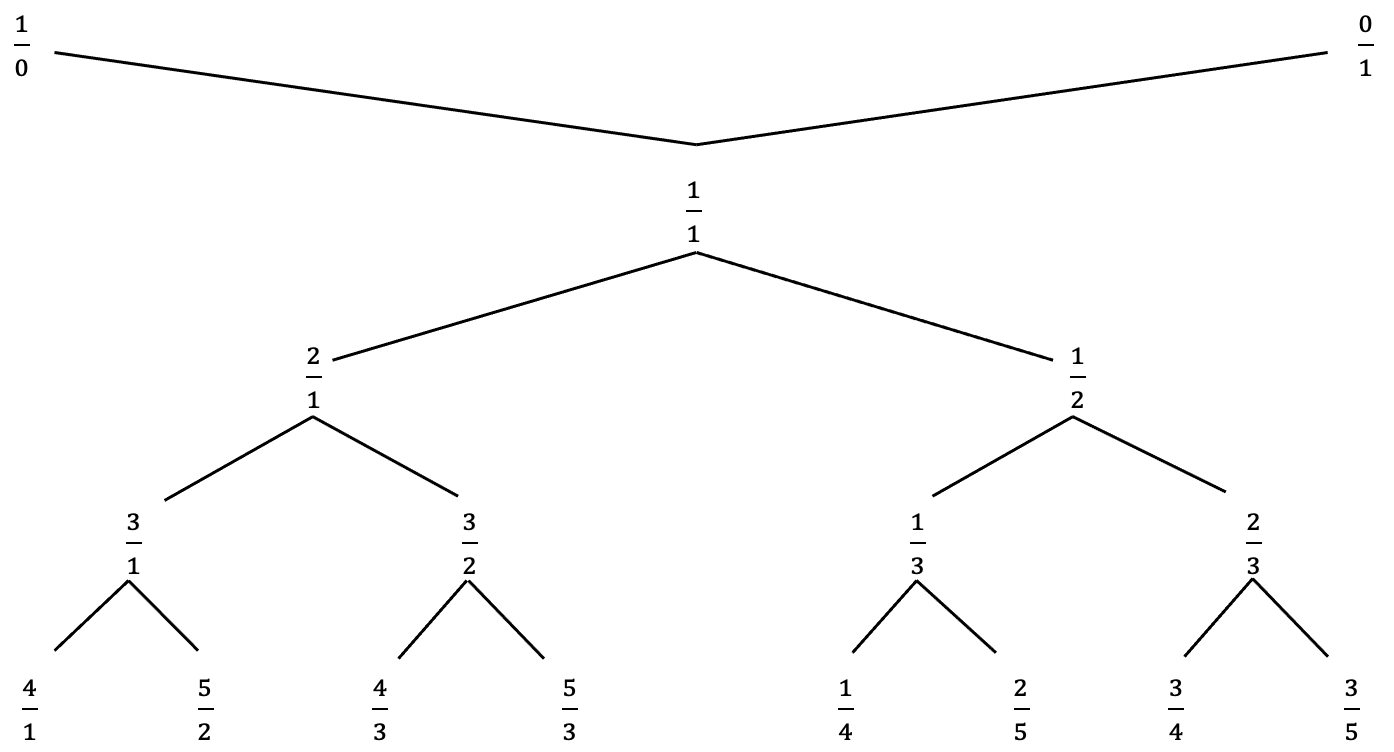}
  \caption{\label{image:Cab-tree-en}Levels 0-4 of the SC-tree}
\end{figure}

\begin{definition}\label{def:e4}[ Mediants operation: $\bigoplus$]
\ \ $\frac{m_1}{n_1} \bigoplus \frac{m_2}{n_2} \triangleq \frac{m_1+m_2}{n_1+n_2}$.
\end{definition}

\begin{definition}\label{def:e5}[Left and Right Branches]The left branch of a vertex $k$, denoted as $L_{k}$, comprises the set of vertices generated by an infinite number of left movements originating from $k$. $L_{k}$ includes the vertex $k$. Similarly, the right branch of a vertex $k$, denoted as $R_{k}$, consists of the set of vertices generated by an infinite number of right movements starting from $k$. $R_{k}$ also includes the vertex $k$.

\end{definition}

\begin{theorem}\label{thm:adde3}
   Let $k$ be a vertex of the SC-tree, then we have: 
   \begin{itemize}
       \item {\romannumeral1}\ \  If $k$ is the right child of vertex $s$, and $s$ is a child vertex of $t$, then $k=s\bigoplus t$. 
       \item {\romannumeral2}\ \ If $k$ is the right child of vertex $s$, and $s$ is a child vertex of $t$, then $L_{k}=\{k_0=k,k_1,\cdots,k_i,\cdots\}$ satisfies $k_i=k_{i-1}\bigoplus t(i\ge 1)$. 
       \item {\romannumeral3}\ \ For integers $n\ge 2, 1\le w\le 2^{n-2}$, if $k$ is the $w$-th vertex in level $n$, then the $(2^{k-2}+w)$-th vertex in level $n$ is $\frac{1}{k}$. 
   \end{itemize}
\end{theorem}

\begin{proof}
\ 

\noindent Figure \ref{e-ab2-tree} displays the section of the S-tree starting with vertex $\frac{ua+vb}{ma+nb}$. Correspondingly, Figure \ref{e-SC-tree2} shows the section of the SC-tree.
\begin{figure}[ht]
  \centering
  \subfigure[\label{e-ab2-tree}A part of S-tree]{\includegraphics[width=0.99\linewidth]{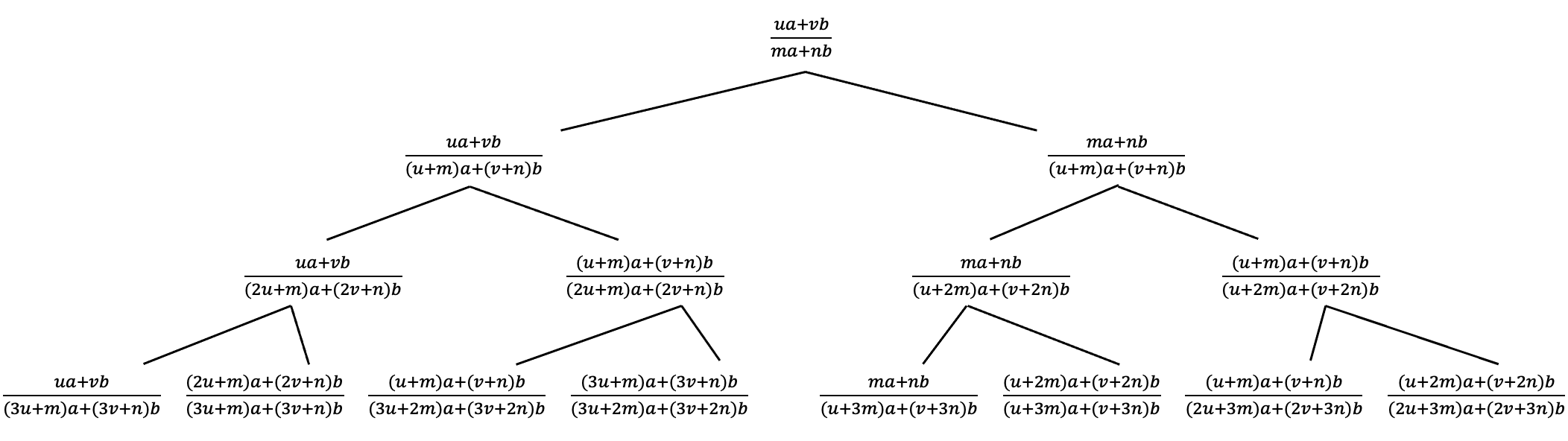}}
  \subfigure[\label{e-SC-tree2}A part of SC-tree]{\includegraphics[width=0.65\linewidth]{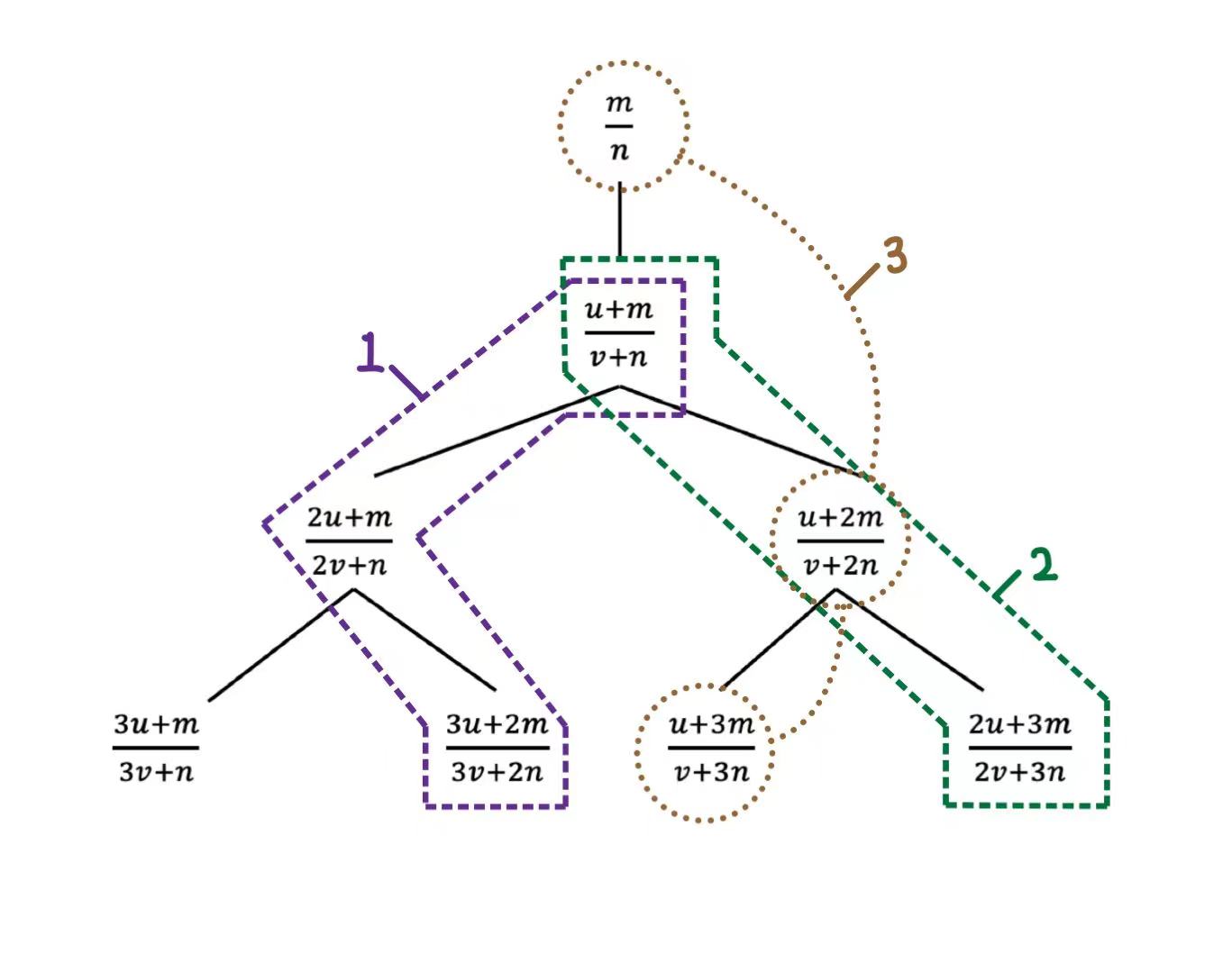}}
  \vspace{1em}
  \caption{\label{e-S-C2} }
\end{figure}

    \noindent {\romannumeral1}. \ The first property is illustrated by the 1 coil and the 2 coil shown in Figure \ref{e-SC-tree2}. \newline 
    {\romannumeral2}. \ The second property is demonstrated by the 3 coil shown in Figure \ref{e-SC-tree2}.  \newline     
    {\romannumeral3}. \ We establish the third property through induction on $n$. 
    
    \noindent The theorem is true for $n=2$, since $\frac{2}{1}$ and $\frac{1}{2}$ are reciprocal.\newline
    Now, let's assume the theorem is true for all integers $2, 3, \cdots , m$.\newline
    If $k$ and $k'$ are the $w$-th and the $(2^{m-2}+w)$-th vertex in level $m$, where $1\le w\le 2^{m-2}$, and $k$ is the child of $s$, $s$ is the child of $t$, and similarly,  if $k'$ is the child of $s'$ and $s'$ is the child of $t'$, then it follows that $k'=\frac{1}{k},s'=\frac{1}{s},t'=\frac{1}{t}$.\newline
    Now, consider the children of $k$ and $k'$: \newline   
    The right child of $k$ is the $2w$-th vertex in level $m+1$, and the right child of $k'$ is the $(2^{m-1}+2w)$-th vertex in the same level. By the first property, they correspond to $k\bigoplus s$ and $k'\bigoplus s'=\frac{1}{k}\bigoplus \frac{1}{s}=\frac{1}{k\bigoplus s}$, respectively. \newline
    The left child of $k$ is the $(2w-1)$-th vertex in level $m+1$, and the left child of $k'$ is the $(2^{m-1}+2w-1)$-th vertex in the same level. By the second property, they correspond to $k\bigoplus t$ and $k'\bigoplus t'=\frac{1}{k}\bigoplus \frac{1}{t}=\frac{1}{k\bigoplus t}$, respectively. Thus, the theorem holds for $n=m+1$.\newline
    Hence our result is true for all $n$.
\end{proof}

\begin{definition}\label{def:e6}[Addable]
For two vertices of the SC-tree $\frac{a}{b}$ and $\frac{c}{d}$, if $\frac{a}{b} \bigoplus \frac{c}{d}$ is also a vertex of the SC-tree, then $\frac{a}{b}$ and $\frac{c}{d}$ are \textbf{addable to each other}.
\end{definition}

  \noindent According to Theorem \ref{thm:adde3}, if $\frac{a}{b}$ and $\frac{c}{d}$ are addable to each other, their relationship must be one of the situations shown by Figure \ref{image:kjw-e}.

\begin{figure}[ht]
  \centering
  \includegraphics[width=0.5\textwidth]{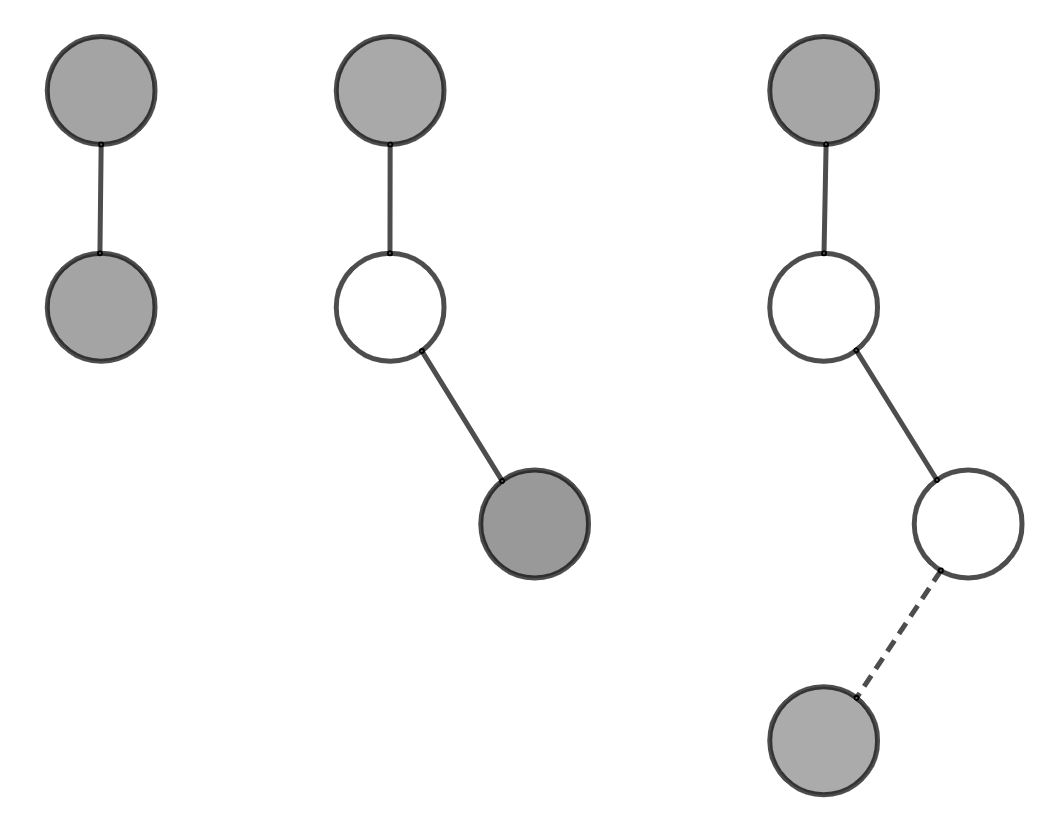}
  \caption{\label{image:kjw-e}The relationship between two vertices that are addable to each other }
\end{figure}  

\begin{theorem}\label{thm:adde3+}
   For two vertices of the SC-tree $\frac{a}{b}$ and $\frac{c}{d}$, they are addable to each other if an only if $bc-ad=\pm 1$.    
\end{theorem}

\begin{proof}

\ 

\noindent "$\Rightarrow$"

     \noindent Check all vertices shown in Figure \ref{image:Cab-tree-en}. Then, we can observe that for levels 3 and above in the tree, every pair of addable vertices $\frac{a}{b}$ and $\frac{c}{d}$ satisfies $bc-ad=\pm 1$. 
     
     \noindent Suppose this holds true for all vertices in levels $n$ and above.  Assuming that the level of $\frac{a}{b}$ is smaller than that of $\frac{c}{d}$, and $\frac{a}{b} \bigoplus \frac{c}{d}=\frac{a+c}{b+d}$ is in level $n+1$((implying $bc-ad=\pm 1$). According to Theorem \ref{thm:adde3}, if $\frac{a}{b}, \frac{c}{d}$, and $\frac{a+c}{b+d}$ exhibit the relationship shown by Figure \ref{jcj-e01} or Figure \ref{jcj-e04} , then $\frac{a+2c}{b+2d}$ is on the SC-tree and $d(a+c)-c(b+d)=ad-bc=\pm 1$. In addition, if the relationship among $\frac{a}{b}, \frac{c}{d}$, and $\frac{a+c}{b+d}$ resembles Figure \ref{jcj-e02} or Figure \ref{jcj-e03} , then $\frac{2a+c}{2b+d}$ is on the SC-tree and $b(a+c)-a(b+d)=bc-ad=\pm 1$. Thus, for two tree vertices $\frac{a}{b}$ and $\frac{c}{d}$ that are addable to each other, $bc-ad=\pm 1$. 

\noindent "$\Leftarrow$"

  \noindent When $a=b=1$, then $bc-ad=\pm 1$ implies that $\frac{c}{d}$ can be one of the following cases: $\frac{1}{0}$, $\frac{0}{1}$, $\frac{1}{2}$, $\frac{2}{1}$, a vertex of $L_{\frac{2}{3}}$, a vertex of $L_{\frac{3}{2}}$. All of these cases above satisfy the addable relationship be shown in Figure \ref{image:kjw-e}.\newline
    According to Theorem \ref{thm:adde3}, now we can only focus on the case that $1\le a < b$ due to the symmetry, that is, the right subtree of the SC-tree rooted at $\frac{1}{1}$. \newline
    Consider the solutions in the positive integers $N_{+}$ for $bc-ad=\pm 1$. When $bc-ad= 1$, the integral solutions grater than 0 are $c=c_0+ta,d=d_0+tb,t=0,1,2,\cdots$, where $0< c_0< a,0< d_0< b$ represent the smallest pair of solutions. In parallel, for $bc-ad= -1$, the corresponding solutions are$c=-c_0+ta,d=-d_0+tb,t=1,2,\cdots$. Therefore,  $c=c_0, d=d_0$ and  $c=a-c_0, d=b-d_0$ are the only two pairs of solutions for $bc-ad=\pm 1$ under the constraint that $0<c<a, 0<d<b$. Furthermore, $\frac{c_0}{d_0}\bigoplus \frac{a-c_0}{b-d_0}=\frac{a}{b}$ demonstrates that $\frac{c_0}{d_o}$ and $\frac{a-c_0}{b-d_0}$ are exactly the vertices of the SC-tree producing $\frac{a}{b}$ with $\bigoplus$. As a result, $\frac{c_0}{d_o}$ ( or $\frac{a-c_0}{b-d_0}$ ) and $\frac{a}{b}$ are addable, as also illustrated in Figure \ref{image:kjw-e}.  Therefore,  for a vertex of the SC-tree $\frac{a}{b}(a\ge 1, b\ge 1)$, if integers $c,d$ satisfy $bc-ad=\pm 1$, then $\frac{a}{b}$ and $\frac{c}{d}$ are addable to each other.\newline
    The result follows.
\end{proof}

\begin{figure*}[ht]
  \centering
  \subfigure[\label{jcj-e01}relationship 1]{\includegraphics[width=0.30\linewidth]{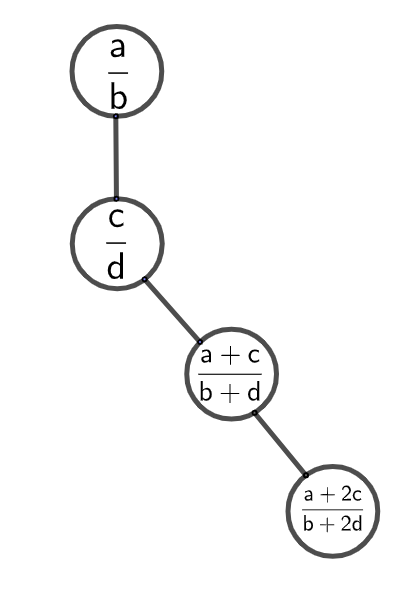}}
  \subfigure[\label{jcj-e02}relationship 2]{\includegraphics[width=0.30\linewidth]{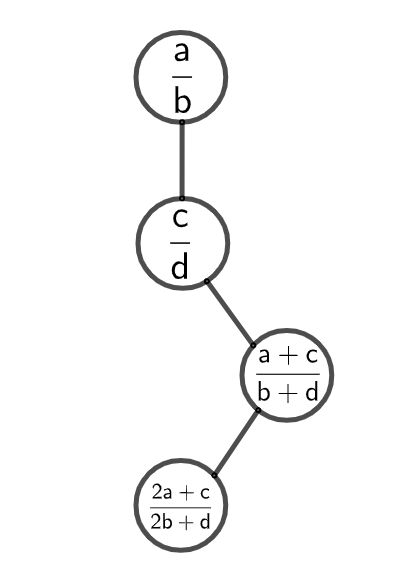}}
  \subfigure[\label{jcj-e03}relationship 3]{\includegraphics[width=0.40\linewidth]{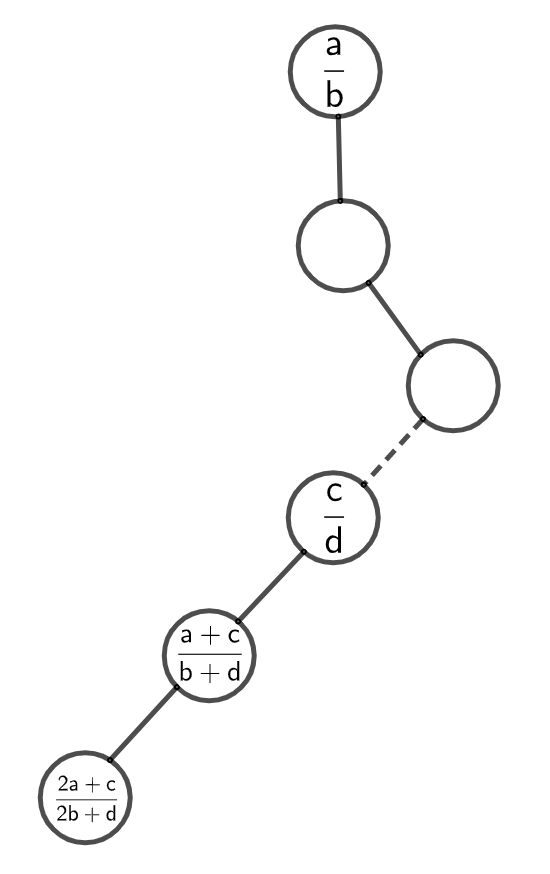}}
  \subfigure[\label{jcj-e04}relationship 4]{\includegraphics[width=0.40\linewidth]{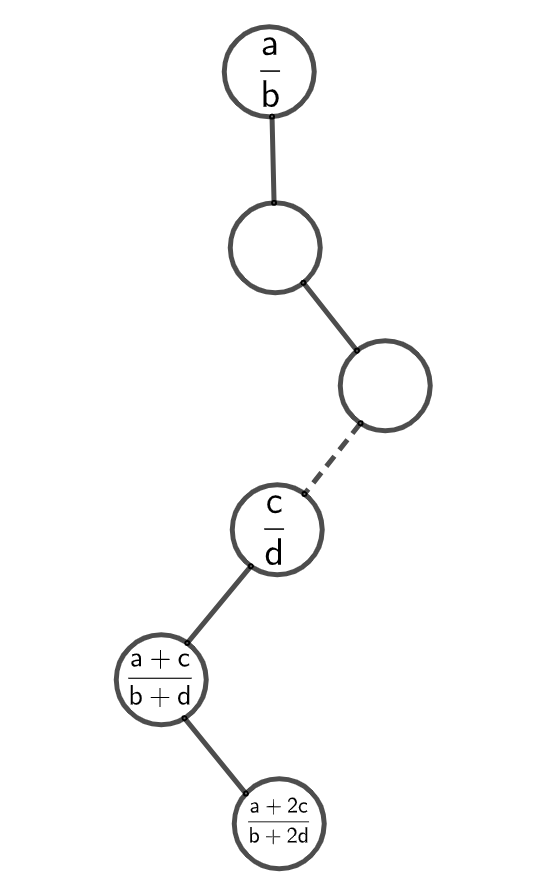}}
  \vspace{1em}
  \caption{\label{jcj-e}The relationship among $\frac{a}{b},\frac{c}{d}$ and $\frac{a+c}{b+d}$}
\end{figure*}

\begin{corollary}\label{coro:e3+}
   All the vertices of the SC-tree( except $\frac{1}{0}$ and $\frac{0}{1}$ ) are reduced rational numbers. 
\end{corollary}

\begin{proof}
\ 

   \noindent According to Theorem \ref{thm:adde3+}, if $\frac{a}{b}(a,b \in N_+)$ is a vertex of the SC-tree, then there must be $c,d\in N_+$ such that $bc-ad=\pm 1$. Consequently, $(a,b)=(c,d)=1$, which means $\frac{a}{b}$ is a reduced rational. 
\end{proof}

\begin{theorem}\label{thm:e3+}
The SC-tree corresponds one-to-one to the set $\mathbf{Q_+}$. 
\end{theorem}

\begin{proof}
\ 

\noindent Accordingly, we only need to consider the right subtree of the SC-tree rooted at $\frac{1}{1}$. 
\begin{itemize}
    \item Every positive rational number is represented by some vertices of the SC-tree. 
\end{itemize}
$\frac{1}{1}$ certainly occurs. Otherwise, let $\frac{a}{b} \in (0,1),(a,b)=1$ be, among all fractions that do not occur, one with the smallest denominator, and among the one with the smallest numerator. Since $(a,b)=1$, there must be integers $0<c_1<a,0<d_1<b$ that makes $c_1b-d_1a=1$, and integers $0<c_2<a,0<d_2<b$ that makes $c_2b-d_2a=-1$, specifically, $c_1+c_2=a,d_1+d_2=b$. Due to the properties of $\frac{a}{b}$ we selected, reduced fractions $\frac{c_1}{d_1}$ and $\frac{c_2}{d_2}$ must be on the SC-tree. Notably, $c_1d_2-d_1c_2=c_1(b-d_1)-d_1(a-c_1)=c_1b-d_1a=1$. As a result, $\frac{c_1+c_2}{d_1+d_2}=\frac{a}{b}$ is a vertex of the SC-tree by Corollary\ref{coro:e3+}, a contradiction. 
\begin{itemize}
    \item No positive rational number occurs at more than one vertex of the SC-tree.
\end{itemize}
$\frac{1}{1}$ certainly occurs once of the tree. Otherwise, let $\frac{a}{b} \in (0,1),(a,b)=1$ be, among all fractions that occur more than once, one with the smallest denominator, and among the one with the smallest numerator. Since $(a,b)=1$, there must be only one pair of integers $0<c_1<a,0<d_1<b$ that makes $c_1b-d_1a=1$, and exactly one pair of integers $0<c_2<a,0<d_2<b$ that makes $c_2b-d_2a=-1$, specifically, $c_1+c_2=a,d_1+d_2=b$. Notice that $c_1d_2-d_1c_2=c_1(b-d_1)-d_1(a-c_1)=c_1b-d_1a=1$. As a result, $\frac{c_1+c_2}{d_1+d_2}$ is the only way to approach $\frac{a}{b}$, which means $\frac{c_1}{d_1}$ and $\frac{c_2}{d_2}$ also occur more than once on the SC-tree. This contradicts the feature of $\frac{a}{b}$ we selected. \newline
So every positive rational appears exactly once. The result follows.  

\end{proof}

In this subsection, we introduced the SC-tree, expanding the range of corresponding rationals beyond the limits of the S-tree, which is limited to the interval $(0,1]$. We begun by examining the origins of the denominators in the S-tree and introduced the SC-tree. Then, we explored the operation of Mediants to outline the properties of the SC-tree and illustrated its growth. Additionally, we clarified the conditions under which two vertices can generate a new reduced fraction using the Mediants operation, as shown in Theorem \ref{thm:adde3+}. Subsequently, we leveraged insights from number theory to demonstrate that every fraction on the SC-tree is reduced, seen by Corollary \ref{coro:e3+}. With these properties in place, we can ultimately conclude that all vertices of the SC-tree can be one-to-one corresponded with rationals from the set $\mathbf{Q_+}$, as illustrated by Theorem \ref{thm:e3+}.
\subsection{Locating vertices of the SC-tree}
\ 

Similar to exploring the S-tree, this subsection mainly focuses on the algorithm for locating vertices of the SC-tree. To simplify the discussion, we return to continued fractions and the 0-1 sequence. The third property in Theorem \ref{thm:adde3} connects the two reciprocal rationals, allowing us to accomplish the locating task by tracking only the locations of rationals not greater than 1.

\begin{theorem}\label{thm:adde4}

For the vertex $k$ in level $n(n\ge 3)$ of the SC-tree, if $k$ is assumed to be the right child of $s$, then we have: 

    \begin{itemize}
        \item {\romannumeral1}\ \  If $s=[0,a_1,a_2,\cdots,a_{m-1}], a_i\ge 1, i = 1, 2, \cdots, m-2, a_{m-1}\ge 2 $, then 
        
        $k=[0,a_1,a_2,\cdots,a_{m-1}-1,2]$.
        
        \item {\romannumeral2}\ \  If $s$ is the left child of $t=[0,a_1,a_2,\cdots,a_{m-1}], a_i\ge 1, i = 1, 2, \cdots, m-2, a_{m-1}\ge 2 $, then
        
        $k=[0,a_1,a_2,\cdots,a_{m-1},2]$. 
        
        Else,if $s$ is the right child of $t=[0,a_1,a_2,\cdots,a_{m-1}], a_i\ge 1, i = 1, 2, \cdots, m-2, a_{m-1}\ge 2 $, then
        
        $k=[0,a_1,a_2,\cdots,a_{m-1}-1,1,2]$.

        \item {\romannumeral3}\ \  If $k=[0,a_1,a_2,\cdots,a_{m-1},a_m], a_i\ge 1, i = 1, 2, \cdots, m-2, a_{m}\ge 2 $, then
        
        $L_{k}=\{k_0=k,k_1,\cdots,k_j,\cdots\}$ satisfies $k_j=[0,a_1,a_2,\cdots,a_{m-1},a_m+j](j\ge 0)$. 
     
    \end{itemize}
\end{theorem}

\begin{proof}
\ 

    \noindent Since the left subtree of the SC-tree rooted at $\frac{2}{1}$ consists of the rationals above 1, the right subtree of the SC-tree with $\frac{1}{2}$ as the root should only be considered.\newline
    After checking all vertices in Figure \ref{image:Cab-tree-en}, we can conclude that the theorem is true for level 4 and above of the SC-tree.\newline
    Now, let's assume the theorem is true for all levels $3, 4, \cdots , g$.\newline
    Consider a vertex $k$ in level $g+1$, where it is the right child of $s$, and $s$ is a child of $t$, while $t$ is a child of $t^*=[0,a_1,a_2,\cdots,a_{m-1}]$. Now, let's examine the continued fraction of $k$ in the context of the four possible arrangements of $s,t,t^*$.
     
\begin{figure}[htbp]
  \centering
  \subfigure[\label{e-tsk01}Case 1]{\includegraphics[width=0.47\linewidth]{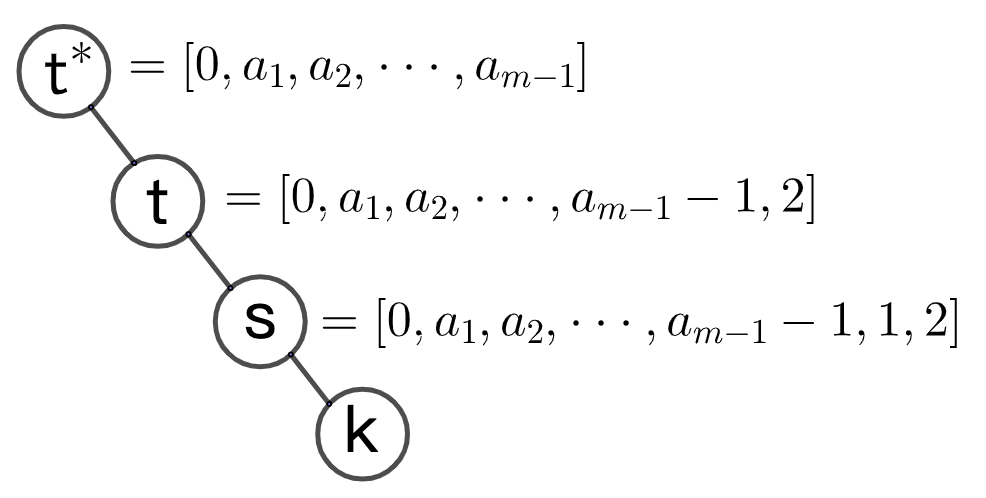}}
  \subfigure[\label{e-tsk02}Case 2]{\includegraphics[width=0.41\linewidth]{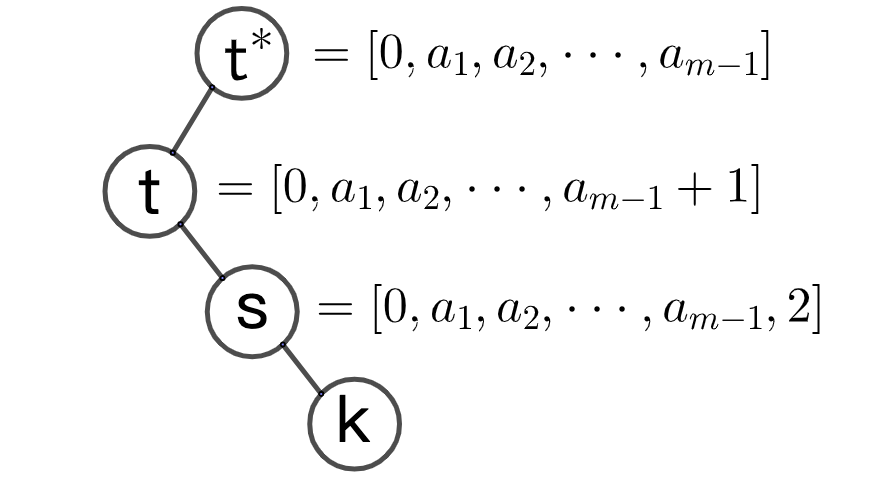}}
  \subfigure[\label{e-tsk03}Case 3]{\includegraphics[width=0.41\linewidth]{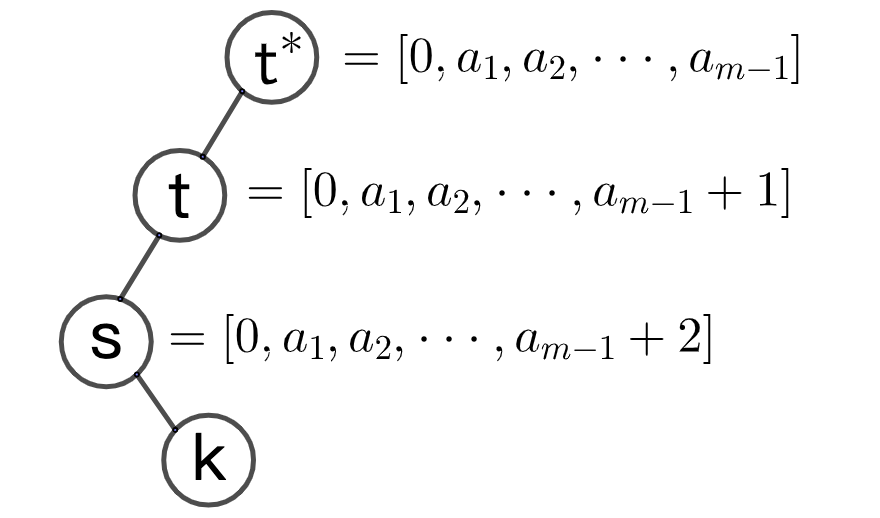}}
  \subfigure[\label{e-tsk04}Case 4]{\includegraphics[width=0.41\linewidth]{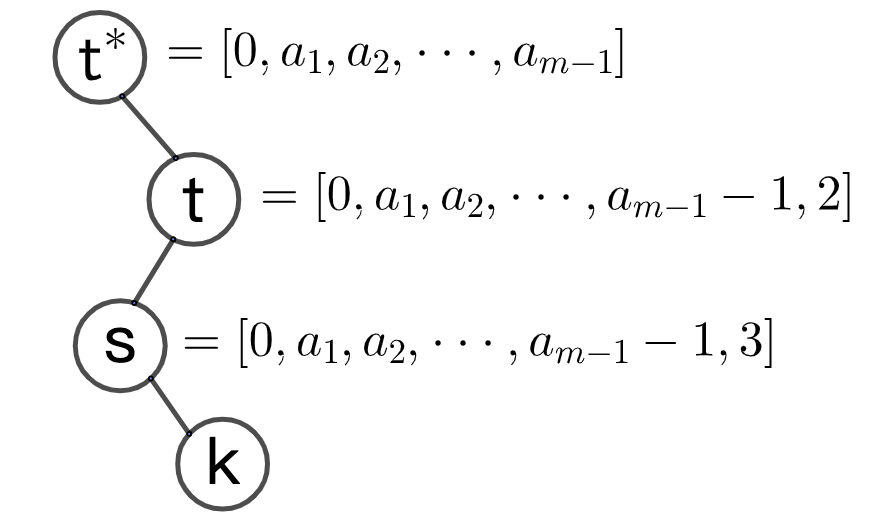}}
  \vspace{1em}
  \caption{\label{e-tsk}Four kinds of arrangement of $s,t,t^*$}
\end{figure}

    \noindent \textbf{\textit{Case} 1}: As shown in Figure \ref{e-tsk01}.\newline    In this case, where $t,s,k$ all belong to $R_{t^*}$, we have $t=[0,a_1,a_2,\cdots,a_{m-1}-1,a_m=2],s=[0,a_1,a_2,\cdots,a_{m-1}-1,1,2]$. From Theorem \ref{thm:adde3}, we know that $k=s\bigoplus t=(t^*\bigoplus t)\bigoplus t=t^*\bigoplus 2t$. Changing the expressions of $t^*$ and $t$ to $t^*=[0,a_1,a_2,\cdots,a_{m-1}-1,1]$ and $t=[0,a_1,a_2,\cdots,a_{m-1}-1,1,1]$, and we find that $k=[0,a_1,a_2,\cdots,a_{m-1}-1,1,1,2]$. Comparing $k$ with $t,s$, we can see the first and the second properties are satisfied. Now, let's explore $L_{k}=\{k_0=k,k_1,\cdots,k_j,\cdots\}(j\ge 0)$. The third property is true when $j=0$. Suppose it holds for all $0 \le j \le l$, $k_j=[0,a_1,a_2,\cdots,a_{m-1}-1,1,1,2+j]$. According to Theorem \ref{thm:adde3}, $k_{l+1}=k_{l}\bigoplus t$, hence $k_{l+1}=[0,a_1,a_2,\cdots,a_{m-1}-1,1,1,2+l,1]=[0,a_1,a_2,\cdots,a_{m-1}-1,1,1,2+(l+1)]$. Therefore,  for $L_{k}$, $k_j=[0,a_1,a_2,\cdots,a_{m-1}-1,1,1,2+j]$ for all $j\ge 0$. \newline
   \textbf{ \textit{Case} 2}: As shown in Figure \ref{e-tsk02}.\newline
    In this case, $t=[0,a_1,a_2,\cdots,a_{m-1}+1],s=[0,a_1,a_2,\cdots,a_{m-1},2]$. From Theorem \ref{thm:adde3}, we know that $k=s\bigoplus t=(t^*\bigoplus t)\bigoplus t=t^*\bigoplus 2t$. Changing the expression of $t$ to $t=[0,a_1,a_2,\cdots,a_{m-1},1]$, we have $k=[0,a_1,a_2,\cdots,a_{m-1},1,2]$. Comparing $k$ with $t,s$, we can see the first and the second properties are satisfied. Now, let's explore $L_{k}=\{k_0=k,k_1,\cdots,k_j,\cdots\}(j\ge 0)$. The third property is true when $j=0$. Suppose it holds that $k_j=[0,a_1,a_2,\cdots,a_{m-1},1,2+j]$ for all $0 \le j \le l$. According to Theorem \ref{thm:adde3}, $k_{l+1}=k_{l}\bigoplus t$, hence $k_{l+1}=[0,a_1,a_2,\cdots,a_{m-1},1,2+l,1]=[0,a_1,a_2,\cdots,a_{m-1},1,2+(l+1)]$. Therefore, for $L_{k}$, $k_j=[0,a_1,a_2,\cdots,a_{m-1},1,2+j]$ for all $j\ge 0$. \newline
    \textbf{\textit{Case} 3}: As shown in Figure \ref{e-tsk03}.\newline
    In this case, $t=[0,a_1,a_2,\cdots,a_{m-1}+1],s=[0,a_1,a_2,\cdots,a_{m-1}+2]$. From Theorem \ref{thm:adde3}, we know that $k=s\bigoplus t$. Changing the expression of $s$ to $s=[0,a_1,a_2,\cdots,a_{m-1}+1,1]$, we find that $k=[0,a_1,a_2,\cdots,a_{m-1}+1,1,1]=[0,a_1,a_2,\cdots,a_{m-1}+1,2]$. Comparing $k$ with $t,s$, we can see the first and the second properties are satisfied. Now, let's explore $L_{k}=\{k_0=k,k_1,\cdots,k_j,\cdots\}$, where$j\ge 0$. The third property is true when $j=0$. Suppose it holds for all $0 \le j \le l$, $k_j=[0,a_1,a_2,\cdots,a_{m-1}+1,2+j]$. According to Theorem \ref{thm:adde3}, $k_{l+1}=k_{l}\bigoplus t$, hence $k_{l+1}=[0,a_1,a_2,\cdots,a_{m-1}+1,2+l,1]=[0,a_1,a_2,\cdots,a_{m-1}+1,2+(l+1)]$. Therefore, for $L_{k}$, $k_j=[0,a_1,a_2,\cdots,a_{m-1}-1,2,2+j]$ for all $j\ge 0$. \newline
    \textbf{\textit{Case} 4}: As shown in Figure \ref{e-tsk04}.\newline
    In this case, $t=[0,a_1,a_2,\cdots,a_{m-1}-1,2],s=[0,a_1,a_2,\cdots,a_{m-1}-1,3]$. From Theorem \ref{thm:adde3}, we know that $k=s\bigoplus t$. Changing the expression of $s$ into $s=[0,a_1,a_2,\cdots,a_{m-1}-1,2,1]$, we find that $k=[0,a_1,a_2,\cdots,a_{m-1}-1,2,1,1]=[0,a_1,a_2,\cdots,a_{m-1}-1,2,2]$. Comparing $k$ with $t,s$, we can see the first and the second properties are satisfied. Now, let's explore $L_{k}=\{k_0=k,k_1,\cdots,k_j,\cdots\}(j\ge 0)$. The third property is true when $j=0$. Suppose it holds for all $0 \le j \le l$, $k_j=[0,a_1,a_2,\cdots,a_{m-1}-1,2,2+j]$. According to Theorem \ref{thm:adde3}, $k_{l+1}=k_{l}\bigoplus t$, hence $k_{l+1}=[0,a_1,a_2,\cdots,a_{m-1}-1,2,2+l,1]=[0,a_1,a_2,\cdots,a_{m-1}-1,2,2+(l+1)]$. Therefore, for $L_{k}$, $k_j=[0,a_1,a_2,\cdots,a_{m-1}-1,2,2+j]$ for all $j\ge 0$. \newline   
    The result follows.
    
\end{proof}

\begin{theorem}\label{thm:adde5}
For a rational number $t\in (0,1)$ represented as $t=[0,a_1,a_2,\cdots,a_k]$ where $a_i\ge 1,a_k\ge 2$, its 0-1 sequence of the SC-tree is $$\overline{1\underbrace{0\cdots0}_{a_1-1}1\underbrace{0\cdots0}_{a_{2}-1}1\underbrace{0\cdots0}_{a_{k-1}-1}\cdots 1\underbrace{0\cdots0}_{a_{k}-2}}.$$
\end{theorem}

\begin{proof}
\ 

\noindent  Since all rationals within $(0,1)$ belong to the subtree rooted at $\frac{1}{2}$, and $\frac{1}{2}=[0,2]$ is in level 2, corresponding to the 0-1 sequence $\overline{1}$. According to Theorem \ref{thm:adde4}, for the right child $m$ of a given vertex, the final digit of its continued fraction must be 2, and all vertices in $L_{m}=\{m_0=m,m_1,m_2,\cdots ,m_i,\cdots \}$ share the same continued fraction length with $m$, with the final digit increasing by 1 as the only difference. Taking the right child of $m_i=[0,a_1,a_2,\cdots ,2+i](a_k>2)$ results in $[0,a_1,a_2,\cdots ,1+i,2]$, which is a longer continued fraction. Coincidentally, $1+i$ is exactly the number of vertices among $m_0=m,m_1,m_2,\cdots ,m_i$. Thus, for $t=[0,a_1,a_2,\cdots,a_k]$, the digit $a_i(1\le i\le k-1)$ represents the number of the vertices shared by the path from $\frac{1}{2}$ to $t$ and the left branch of $[0,a_1,a_2,\cdots,a_{i-1},2]$. While 1 should be subtracted when the process goes to $a_k$. Accordingly, when $t=[0,a_1,a_2,\cdots,a_k]$, its 0-1 sequence of the SC-tree should be $\overline{1\underbrace{0\cdots0}_{a_1-1}1\underbrace{0\cdots0}_{a_{2}-1}1\underbrace{0\cdots0}_{a_{k-1}-1}\cdots 1\underbrace{0\cdots0}_{a_{k}-2}}$. 
    
\end{proof}

\begin{corollary}\label{coro:adde2}
 For a rational number $t\in (1,+\infty)$ represented as  $t=[a_1,a_2,\cdots,a_k]$ where $a_i\ge 1,a_k\ge 2$,its 0-1 sequence of the SC-tree is $$\overline{0\underbrace{0\cdots0}_{a_1-1}1\underbrace{0\cdots0}_{a_{2}-1}1\underbrace{0\cdots0}_{a_{k-1}-1}\cdots 1\underbrace{0\cdots0}_{a_{k}-2}}.$$
\end{corollary}

\begin{proof}
\ 

\noindent The symmetry of the SC-tree is evident from the third property in Theorem \ref{thm:adde3}, and this result can be concluded by utilizing Theorem \ref{thm:adde5} mentioned above. 
\end{proof}

\begin{example}\label{examadde04}
Since $\frac{7}{16}=[0,2,3,2],\frac{21}{11}=[2,1,1,1,3]$, by Theorem \ref{thm:adde5} and Corollary \ref{coro:adde2}, they corespond the 0-1 sequences 101001 and 0011110 respectively. Hence, $\frac{7}{16}$ is the $2^5+2^3+2^0+1=42$nd vertex in level 7 of the SC-tree, while $\frac{21}{11}$ is the $2^4+2^3+2^2+2^1+1=31$st vertex in level 8. 
\end{example}

In this subsection, we accomplished the locating task for all vertices on the SC-tree, as shown by Theorem \ref{thm:adde5} and Corollary \ref{coro:adde2} .  These algorithms naturally reveal the relationship between a vertex and its predecessors through continued fraction expansion.

In this section, we constructed the S-tree and the SC-tree, analyzing their properties. Theorem \ref{thm:e1} and \ref{thm:e3+} demonstrate how the the S-tree and the SC-tree offer an intuitive way to establish that rationals are countable.  We also investigated the situations when two vertices are addable, as shown by Theorem \ref{thm:adde3+}, to illustrate that the SC-tree corresponds one-to-one to the reduced rationals in $Q_+$. Additionally, Theorem \ref{thm:en2} and Corollary \ref{coro:adde2} utilize continued fraction expansion and 0-1 sequences to determine the locations of the vertices of both the S-tree and the SC-tree. 


\section{The linkages of the four trees}
\ 

In the previous section, we introduced the S-tree and SC-tree, along with some properties inferred from their definitions. Theorem \ref{thm:en2} and Corollary \ref{coro:adde2} reveal the construction patterns within these trees through the use of the 0-1 sequences and continued fractions. Notice that both the SC-tree and the Stern-Brocot tree correspond one-to-one to the set $Q_+$, so there must be a way to link them together. Building on these results and tools we have obtained, in this section, we will naturally demonstrate the connection between the SC-tree and the Stern-Brocot tree, as well as the partial correspondence between the S-tree and the Calkin-Wilf tree. 
\subsection{The SC-tree and the Stern-Brocot tree}
\ 

 Upon reviewing the Stern-Brocot tree in Figure \ref{Stern-Brocot-image}, we observe that the SC-tree shares the same vertices as the Stern-Brocot tree within all levels, with the only difference being the order of the vertices. In this subsection, we will focus on elucidating the relationship between the SC-tree and the Stern-Brocot tree using the 0-1 sequence defined by Definition \ref{def:e2}.

\begin{theorem}\label{thm:en3}
 Let $\overline{E}$ and $\overline{F}$ be two 0-1 sequences of the same length, after adding '$1$' as the first digit, the value of  $\overline{1E}$ is $\frac{\Sigma_{i=1}^{t}2^{r_i-1}(2^{k_i+1}-2)}{2}$ and $\overline{1F}$ is $\lfloor\frac{\Sigma_{i=1}^{t}2^{r_i-1}(2^{k_i}+1)}{2}\rfloor$ , where $t,r_i,k_i\in Z_{+}$ and $r_{i+1}-(r_i+k_i)\geq 0$. Then the vertex ordered by $\overline{E}$ of the Stern-Brocot tree is the same as the vertex ordered by $\overline{F}$ of the SC-tree.
\end{theorem}

\begin{proof}
\ 

\noindent Let $B$ be a vertex of the Stern-Brocot tree. Since every vertex (except the very beginnings $\frac{0}{1}$ and $\frac{1}{0}$) on the Stern-Brocot tree is the direct sum of two neighboring terms from the Stern-Brocot sequence\cite{Bates2010}, we can rewrite those vertices as $(B_l,B,B_r)$, where $B_l$ is its left addition term and $B_r$ is the right one ( shown in Figure \ref{SB-lr-tree-en}). Similarly, rewrite the vertices of the S-tree with the root vertex $\frac{a}{b}$ into a three-term form. Here the mediums is the denominator of the child of the original vertex, the left term and the right term are the numerators of the left child and the right child of that original vertex respectively, as shown in Figure \ref{SC-lr-tree-en}. 

\begin{figure}[ht]
  \centering
  \subfigure[\label{SB-lr-tree-en}Rewrite the vertices of the Stern-Brocot tree]{\includegraphics[width=0.6\linewidth]{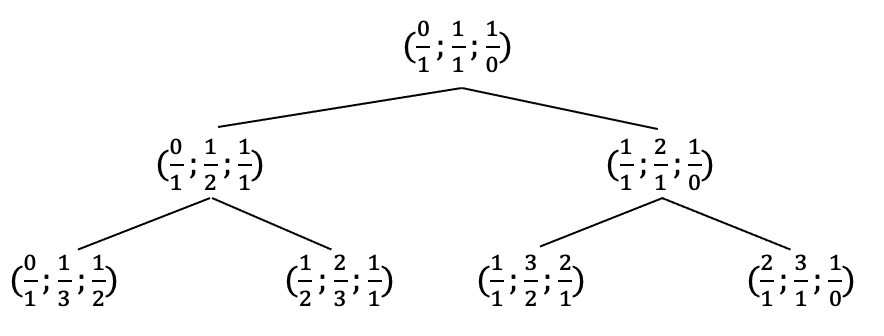}}
  \subfigure[\label{SC-lr-tree-en}Rewrite the vertices of the S-tree]{\includegraphics[width=0.9\linewidth]{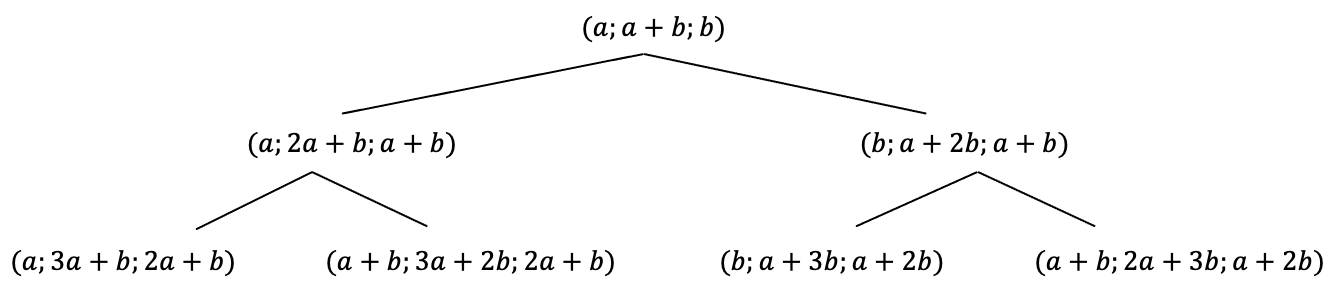}}
  \subfigure[\label{SC-11-tree-en}Replace a and b with  $\frac{1}{0}$ and  $\frac{0}{1}$ ]{\includegraphics[width=0.6\linewidth]{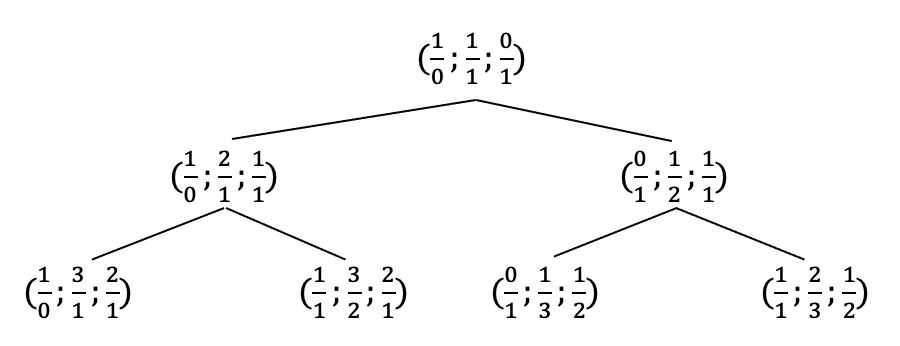}}
  \vspace{1em}
  \caption{\label{lr-tree-en}Rewrite the vertices}
\end{figure}

\noindent It is clear that when we use $\frac{1}{0}$ and  $\frac{0}{1}$ to replace $a$ and $b$ respectively , as shown in Figure \ref{SC-11-tree-en}, all the medium terms form the SC-tree. We can denote a vertex of the tree shown in Figure \ref{SC-11-tree-en} as $(D_l,D,D_r)$.   

\noindent Based on the growing pattern of the Stern-Brocot tree and the S-tree, the children of $(B_l,B,B_r)$ are $(B_l;B_l \bigoplus B;B)$ and $(B;B \bigoplus B_r;B_r)$, while for $(D_l,D,D_r)$, its children are $(D_l;D_l \bigoplus D;D)$ and $(D_r;D \bigoplus D_r;D)$. Consequently, the relationships among them can be summarized as follows:

\noindent \textbf{Case 1}: If $(B_l,B,B_r)$ is the same as $(D_l,D,D_r)$, then their left children are identical, and their right children are in reverse order.

\noindent \textbf{Case 2}: If $(B_l,B,B_r)$ and $(D_l,D,D_r)$ are in reverse order, then the left child of $(B_l,B,B_r)$ matches the right child of $(D_l,D,D_r)$, while the right child of $(B_l,B,B_r)$ and the left child of $(D_l,D,D_r)$ are in reverse order.

\noindent For any positive reduced fraction $W$ (expect $\frac{1}{1}$), let the 0-1 sequence ordering $W$ of the Stern-Brocot tree be denoted as $\overline{E}$ and the 0-1 sequence ordering $W$ of the SC-tree be denoted as $\overline{F}$. Since only the medium terms are exactly equivalent to the vertices of the two trees, and their roots are in reverse order, coordinate with the second situation, we have: 

\begin{itemize}
    \item The first digit  of $\overline{E}$ and $\overline{F}$ are different. 
    \item If the final digit of $\overline{E}$ is '$0$', then \textbf{\textit{add '$0$' and '$1$'}} to the right side of $\overline{F}$ can construct the 0-1 sequence on the SC-tree ordering the $W$'s left child and right child respectively; else, if the final digit of $\overline{E}$ is '$1$',  \textbf{\textit{add '$1$' and '$0$' }} achieves the same result. 
\end{itemize}

\noindent So $\overline{E}$ and $\overline{F}$ share the same length, and we can obtain $\overline{F}$ from the digits of $\overline{E}$ using the following recursive algorithm:  

\textit{Step 1: Change the first digit}; 

\textit{Step 2 : Each time '$1$' is encountered in $\overline{E}$, change the number immediately behind this '$1$' while leaving the number in place if '$0$' is encountered. Continue this process until reaching the final digit of $\overline{E}$. } 

\noindent For example, if the 0-1 sequence ordering $W$ of the Stern-Brocot tree is $\overline{E}=\overline{010011}$, then the 0-1 sequence of $W$ on the SC-tree is $\overline{F}=\overline{111010}$.

\noindent Adding '$1$' as the very first digit and '$0$' as the very last digit to the 0-1 sequence of the two trees aligns the original first and final digits with \textit{Step 2}. This adjustment allows us to express $\overline{1E0}$ as the sum of several 0-1 sequences  denoted as $\overline{\underbrace{1\cdots1}_{m}0\underbrace{0\cdots0}_{n}}$ with  a value of $2^{n}(2^{m+1}-2)$ ($m\geq 1, n\geq 0$). When $\Sigma_{i=1}^{t}2^{r_i-1}(2^{k_i+1}-2)$ represents the value of the adjusted 0-1 sequence $1E0$, then $\Sigma_{i=1}^{t}2^{r_i-1}(2^{k_i}+1)$ must correspond to $\overline{1F0}$ or $\overline{1F1}$. Adding '$0$' as the final digit to a 0-1 sequence is equivalent to multiply its value by 2, and then dividing the result by 2 while eliminating its decimal part, effectively eliminating the adjustment of add '$0$'. To ensure continuity and avoid accidental splitting of consecutive '$1$'s ( like dividing $\overline{1110}$ into the sum of $\overline{1000}$ and $\overline{110}$, which is not the expected way ), we introduce the condition $r_{i+1}-(r_i+k_i)\geq 0$.

\noindent In summary,  when we add '$1$' on the far left to $\overline{E}$ and $\overline{F}$( two  0-1 sequences of the same length ), the value of $\overline{1E}$ is $\frac{\Sigma_{i=1}^{t}2^{r_i-1}(2^{k_i+1}-2)}{2}$, and  the value of $\overline{1F}$ is $\lfloor\frac{\Sigma_{i=1}^{t}2^{r_i-1}(2^{k_i}+1)}{2}\rfloor$ , where $t,r_i,k_i\in Z_{+}$, $r_{i+1}-(r_i+k_i)\geq 0$, and $\lfloor x \rfloor$ represents the greatest integer that is less than or equal to $x$. Then, the vertex of the Stern-Brocot tree indicated by $\overline{E}$, is exactly the vertex of the SC-tree indicated by $\overline{F}$.  

\noindent The result follows.
    
\end{proof}

\begin{example}\label{exam:02-en}
   
   Referring to the analysis in Example \ref{exam:01-en}, the 63rd vertex in level 70 of the Stern-Brocot tree corresponds to the 0-1 sequence $\overline{\underbrace{0\cdots 0}_{70-1-(5+1)}111110}$. Adding '$1$' and '$0$' as the very first and very last digits results in the adjusted 0-1 sequence: $$\overline{1\underbrace{0\cdots 0}_{63}1111100}=2^{69}(2^2-2)+2(2^6-2).$$ According to Theorem \ref{thm:en3}, this adjusted sequence corresponds to the same vertex in the SC-tree, indicated as $$\lfloor\frac{2^{69}(2^2-2)+2(2^6-2)}{2}\rfloor=2^{69}+2^{68}+2^5+2^0.$$ Removing the first digit leaves $\overline{1\underbrace{0\cdots 0}_{62}100001}$ as the real location indicator. For $2^{68}+2^5+2^0+1=2^{68}+2^5+2^1, 69+1=70$, accordingly, the mentioned vertex is the No.$2^{68}+2^5+2^1$ vertex in level $70$ on the SC-tree.
   
\end{example}

The one-to-one correspondence between the SC-tree and the Stern-Brocot tree is illustrated by Theorem \ref{thm:en3}, utilizing the 0-1 sequences. In order to demonstrate the relationship among the four trees mentioned before, we need to involve the Calkin-Wilf tree. While previous papers, including those by Bruce Bates, have explored the correspondence between the Stern-Brocot tree and the Calkin-Wilf tree, we choose to simply establish this linkage by locating the vertices of the S-tree in the Calkin-Wilf tree, as shown in the next subsection.

\subsection{Locating vertices of S-tree in the Calkin-Wilf tree}
\ 

The Calkin-Wilf tree, as shown in Figure \ref{Calkin-Wilf-image}, is rooted at $\frac{1}{1}$, and any rational number expressed in simplest terms as the fraction $\frac{a}{b}$ has its two children, $\frac{a}{a+b}$ and  $\frac{a+b}{b}$. The similarities between the Calkin-Wilf tree and the S-tree can be found in Definition \ref{def:e1}. The vertices of the Calkin-Wilf tree correspond one-to-one to all reduced rational numbers in $(0,+\infty)$\cite{Calkin}, meaning that every fraction growing on the S-tree can be located in the Calkin-Wilf tree. Our main goal in this subsection is to illustrate the relationship between the Calkin-Wilf tree and the S-tree, which might not be immediately apparent when comparing Figure \ref{Calkin-Wilf-image} and Figure \ref{S-tree-e1}.

When numbering the levels of the Calkin-Wilf tree starting from $0$, with $\frac{1}{1}$ as the root, and ordering the vertices using 0-1 sequences based on Definition \ref{def:e2}, a vertex from the S-tree can be located on the Calkin-Wilf tree using the following algorithm.

\begin{theorem}\label{thm:add2-en}

Let the binary expansion of $N_s-1$ be  $\Sigma_{i=1}^{k}2^{r_i}$ ( for $u<v$, $r_u>r_v$ ). Then, the $N_s$-th vertex in the $M$-th level ($M\ge 1, N_s\le 2^{M-1}$) of the S-tree corresponds the $N_c$-th vertex in the $M$-th level of the Calkin-Wilf tree, where $N_c=\left\{
\begin{aligned}
    &1+\Sigma_{j=1}^{\frac{k}{2}}(2^{r_{2j-1}+1}-2^{r_{2j}+1}),\ \ 2\mid {k} \\
    &1+2^{M}-2^{r_1}+\Sigma_{j=1}^{\frac{k-1}{2}}(2^{r_{2j}+1}-2^{r_{2j+1}+1}),\ \ \ \ \ 2\nmid {k}\\
\end{aligned}
    \right.$. 
\end{theorem}

\begin{proof}
\ 

\noindent According to the growth rules of both the S-tree and the Calkin-Wilf tree, following properties can be observed: 

    \begin{itemize}    
    \item \textit{ If the vertex $A$ of the S-tree and the vertex $B$ of the Calkin-Wilf tree are identical, then they have the same left child and their right child are reciprocal to each other. } 
    \item \textit{If the vertex $A$ of the S-tree and the vertex $B$ of the Calkin-Wilf tree are reciprocal to each other, then the right child of $A$ is the same as the left child of $B$, and the other child of $A$ and $B$ are reciprocal to each other. }
    \end{itemize}

\noindent Thus, all the fractions appear in the $M$-th level of the S-tree are also present in the $M$-th level of the Calkin-Wilf tree. The leftmost branch of the S-tree is an exact match with that of the Calkin-Wilf tree. When examining the vertices with 0-1 sequences that contain '$1$'s, we denote $\overline{X_a}$ as the 0-1 sequence of vertex $a$ on the S-tree, $\overline{Y_b}$ as the 0-1 sequence of vertex $b$ on the Calkin-Wilf tree.  \newline
According to the second property, when $\overline{X_a}=\overline{\underbrace{0\cdots0}_{n}1}, \overline{Y_b}=\overline{0\underbrace{0\cdots0}_{n}1}$ , it implies that vertices $a$ and $b$ are reciprocal, and they both occupy the 2nd position in level $M-r_1$ ( Note that $\frac{1}{2}$ is not included when ordering the vertices of the S-tree while $\frac{1}{2}$ is represented by '$0$' of the Calkin-Wilf tree. Therefore, $\overline{Y_b}$ has a first extra '$0$' compared to $\overline{X_a}$  ). Additionally, when 
$$\overline{X_a}=\overline{\underbrace{0\cdots0}_{n_1}1\underbrace{0\cdots0}_{n_2}1}$$  
$$ \overline{Y_b}=\overline{0\underbrace{0\cdots0}_{n_1}1\underbrace{1\cdots1}_{n_2}0}$$ we would see $a=b$.  \newline
The $N_s$-th ( $N_s=2^{r_1}+2^{r_2}+\cdots+2^{r_k}$ ) vertex in the $M$-th level of the S-tree is represented by the 0-1 sequence: 
$$\overline{X_a}=\overline{\underbrace{0\cdots0}_{M-r_1-2}1\underbrace{0\cdots0}_{r_1-r_2-1}10\cdots01\underbrace{0\cdots0}_{r_{i-1}-r_{i}-1}10\cdots01\underbrace{0\cdots0}_{r_{k-1}-r_{k}-1}1\underbrace{0\cdots0}_{r_k}}.$$
\noindent Similarly, we prove our result by induction on $k$. \newline
When $2\mid {k}$, by changing $\overline{X_a}$'s sequence segments that match $\overline{1\underbrace{0\cdots0}_{r_{i-1}-r_{i}-1}1}$ into $\overline{1\underbrace{1\cdots1}_{r_{k-1}-r_{k}-1}0}$ and adding '$0$' as the very first digit, we can obtain $\overline{Y_a}$:  

$$\overline{Y_a}=\overline{0\underbrace{0\cdots0}_{M-r_1-2}1\underbrace{1\cdots1}_{r_1-r_2-1}00\cdots01\underbrace{1\cdots1}_{r_{i-1}-r_{i}-1}00\cdots01\underbrace{1\cdots1}_{r_{k-1}-r_{k}-1}0\underbrace{0\cdots0}_{r_k}}.$$ Hence the value of $\overline{Y_a}$ is $\Sigma_{j=1}^{\frac{k}{2}}(2^{r_{2j-1}+1}-2^{r_{2j}+1})$.  \newline
When $2\nmid {k}$, we can firstly get $\overline{Y_{\frac{1}{a}}}$ by the similar treatment, and have the result:

    $$\overline{Y_{\frac{1}{a}}='0\underbrace{0\cdots0}_{M-r_1-2}1\underbrace{1\cdots1}_{r_1-r_2-1}00\cdots01\underbrace{1\cdots1}_{r_{i-1}-r_{i}-1}00\cdots01\underbrace{1\cdots1}_{r_{k-2}-r_{k-1}-1}00\cdots01\underbrace{1\cdots1'}_{r_k}}.$$
\noindent As the Calkin-Wilf tree certainly owns symmetry, $\overline{Y_{a}}$ could get by interchanging '$0$' and '$1$' from $\overline{Y_{\frac{1}{a}}}$. That is, 
    $$\overline{Y_{a}}=\overline{1\underbrace{1\cdots1}_{M-r_1-2}0\underbrace{0\cdots0}_{r_1-r_2-1}11\cdots10\underbrace{0\cdots0}_{r_{i-1}-r_{i}-1}11\cdots10\underbrace{0\cdots0}_{r_{k-2}-r_{k-1}-1}11\cdots10\underbrace{0\cdots0}_{r_k}}.$$
\noindent In this case, the value of $\overline{Y_a}$ is $2^{M}-2^{r_1}+\Sigma_{j=1}^{\frac{k-1}{2}}(2^{r_{2j}+1}-2^{r_{2j+1}+1})$.  \newline
The result follows.
    
\end{proof}

\begin{example}\label{exam:03-en}
   Refer to Example \ref{exam:01-en}, the 63rd vertex in level 70 of the S-tree is $\frac{328}{853}$.  Let $\overline{X_{\frac{328}{853}}}$ denote the 0-1 sequence of this fraction of the S-tree, while and $\overline{Y_{\frac{328}{853}}}$ represents that of the Calkin-Wilf tree. Since $63=1+2^5+2^4+2^3+2^2+2^1, 5\nmid 2$, according to Theorem \ref{thm:add2-en}, the value of $\overline{Y_{\frac{328}{853}}}$ is $2^{70}-2^{5+1}+2^{4+1}-2^{3+1}+2^{2+1}-2^{1+1}=2^{70}-2^{6}+2^{4}+2^{2}$. Thus, $\frac{328}{853}$ corresponds to the $(1+2^{70}-2^{6}+2^{4}+2^{2})$-th vertex in level $70$ of the  Calkin-Wilf tree. Actually, as $\overline{X_{\frac{328}{853}}}=\overline{\underbrace{0\cdots 0}_{70-1-(5+1)}111110}$, applying the steps mentioned in the analysis of Theorem \ref{thm:add2-en} reveals that $\overline{Y_{\frac{328}{853}}}=\overline{1\underbrace{1\cdots 1}_{70-1-(5+1)}010100}$, which still translates to the same value,  $2^{70}-2^{6}+2^{4}+2^{2}$.   
\end{example}

As demonstrated in Example \ref{exam:03-en}, Theorem \ref{thm:add2-en} provides a method for establishing a connection between the S-tree and the Calkin-Wilf tree. What's noteworthy is that this connection is based solely on 0-1 sequences, allowing for the precise location of specific rational numbers. This result is non trivial as the arrangement of rationals within the S-tree is different from that in the Calkin-Wilf tree, which make the pattern less obvious. 

 \begin{figure}[ht]
  \centering
  \includegraphics[width=0.38\textwidth]{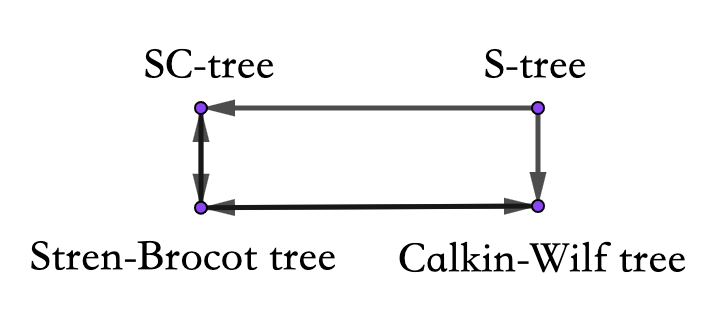}
  \caption{\label{image:looop}The linking loop among the four trees}
\end{figure}

 All connections among the S-tree, the Calkin-Wilf tree, the SC-tree and the Stern-Brocot tree have been established by Theorem \ref{thm:en3} and \ref{thm:add2-en}. The intricate web connecting these four trees, as depicted in Figure \ref{image:looop}, can be illuminated by simply calling the 0-1 sequences. To be specific, Theorem \ref{thm:en3} demonstrates a one-to-one mapping of vertices from the SC-tree to the Stern-Brocot tree. Meanwhile, Definition \ref{def:e3} and Theorem \ref{thm:add2-en} illustrate that the S-tree serves as the crucial bridge, encapsulating the essential details within the other two trees. Building upon these connections, the linkage between the Stern-Brocot tree and the Calkin-Wilf tree becomes intuitively evident, offering a delightful insight into their interrelationship.

\section{Application}
\ 

Building on the results obtained for the four trees, this subsection will highlight how these properties serve as a theoretical foundation for potential applications.  

\subsection{Fibonacci sequence}
\ 

Observe that in the S-tree, the right child of $\frac{m}{n}$ is $\frac{n}{m+n}$, indicating that starting from $\frac{1}{1}$ and consistently choosing the right child leads to vertices with adjacent terms of the Fibonacci sequence as their numerator and denominator. In this subsection,  we leverage the properties we've established to demonstrate the role of the S-tree in relation to the Fibonacci sequence. 

\begin{corollary}\label{coro:02-en} The $n$-th term of the Fibonacci sequence (begins with $0$ and $1$) corresponds the numerator of the continued fraction $[0,\underbrace{1,\cdots,1}_{n}]$.
\end{corollary}

\begin{proof}
\ 

   \noindent The $n$-th term of the Fibonacci sequence is located at the $2^{n-1}$-th vertex in level $n-1$ of the S-tree. According to Theorem \ref{thm:e1},  $M=n-1,N-1=2^{n-1}-1=2^{n-2}+2^{n-3}+\cdots+2^{1}+2^{0}$. Therefore, $k=n-2,p_{i-1}-p_i=1,i=1,2,\cdots,k$, and the continued fraction representation is $[0,1,\underbrace{1,\cdots,1}_{n-2},1]=[0,\underbrace{1,\cdots,1}_{n}]$. 
\end{proof}

When $n$ is not very large,  Corollary \ref{coro:02-en} offers a convenient method that avoids complex calculations involving irritations when compared to the well-known formula $F_n=\frac{1}{\sqrt{5}}\left(\left(\frac{1+\sqrt{5}}{2}\right)^n+\left(\frac{1-\sqrt{5}}{2}\right)^n\right)$. Actually, Corollary \ref{coro:02-en} provides an alternative explanation for the continued fraction representation of the Fibonacci sequence, as $\frac{\sqrt{5}-1}{2}=[0,1,1,1,\cdots]$. Moreover, when the Fibonacci sequence starts with other positive integers $F_1,F_2(F_1<F_2)$, calculation method similar to Corollary \ref{coro:02-en} can be applied, starting from the vertex $\frac{F_1}{F_2}$ of the S-tree.

\subsection{Trigonometric function}
\ 

In this subsection, we explore another application of the S-tree. This topic was prompted by the problem 2 of 24th USAMO:
 
\textit{ A calculator is broken so that the only keys that still work are the $\, \sin, \; \cos,$ $\tan, \; \sin^{-1}, \; \cos^{-1}, \,$ and $\, \tan^{-1} \,$ buttons. The display initially shows 0. Given any positive rational number $\, q, \,$ show that pressing some finite sequence of buttons will yield $\, q$. Assume that the calculator does real number calculations with infinite precision. All functions are in terms of radians.}\cite{GP}

Notice that, $ \sin{\left(\arctan{\left(\sqrt{\frac{a}{b}}\right)}\right)}=\sqrt{\frac{a}{a+b}}$, and $ \cos{\left(\arctan{\left(\sqrt{\frac{a}{b}}\right)}\right)}=\sqrt{\frac{b}{a+b}}$ can be concluded for any right triangle with base $\sqrt{a}$ and height $\sqrt{b} ( a,b\in Z_{+},b\geq a )$.  Consequently, $\frac{a}{a+b}$ and $\frac{b}{a+b}$ are the children of $\frac{a}{b}$. Thus, the answer to the problem mentioned above can be obtained using the S-tree.

\begin{corollary}\label{coro:03-en}
\ \ 
For any given reduced rational number $q\in (0,1)$, when the 0-1 sequence ordering $q^2$ of the S-tree is rewritten using a composite function, where '$0$' is replaced with $\sin{\left(\arctan{(x)}\right)}$ and '$1$' is replaced with $\cos{\left(\arctan{(x)}\right)}$, the resulting composite function $F(x)$ satisfies the property that $F(\cos{0})=q$.

\end{corollary}
Corollary \ref{coro:03-en} demonstrates a method to obtain any reduced rational number within the interval $(0,1)$. The existence of the composite function $F(x)$ can be deduced from Theorem \ref{thm:e1}, the finiteness of $F(x)$ is evident from the limited length of the branch of the S-tree that connects $\frac{1}{1}$ and $q^2$. Additionally, the result of $\cos{0}$ is $1$. Regarding reduced rational numbers greater than $1$, we have ${\frac{b}{a}}=\tan{\left(\arccos{\left(\sin{\left(\arctan{\left(\frac{a}{b}\right)}\right)}\right)}\right)}$. Therefore, all reduced rational numbers in $(0,+\infty)$ can be generated by pressing a finite sequence of the provided buttons, thereby confirming the proposition of the problem.

\section{Conclusion}
\ 

This paper introduces the construction of the S-tree and the SC-tree, utilizing continued fractions and 0-1 sequences to explore the relationships between these two trees and the Stern-Brocot tree and the Calkin-Wilf tree. To begin, we employ elementary number theory and descriptions of addable vertices to derive the properties of both the S-tree and the SC-tree. As demonstrated in Theorem \ref{thm:e1} and Theorem \ref{thm:e3+}, we establish their one-to-one correspondence with reduced rational numbers within the intervals $(0,1]$ and $(0,+\infty)$, respectively. In other words, both the S-tree and the SC-tree provide an intuitive explanation for the countability of the set $Q$. 

By using '$0$' to represent left and '$1$' to represent right, we can describe any path to a given vertex in the binary tree as a sequence of 0-1. This allows us to illustrate the intrinsic connection between the structure of continued fractions of a given rational number and its corresponding 0-1 sequence in the trees mentioned earlier. Through this process, we successfully developed locating algorithms for the S-tree and the SC-tree, as outlined in Theorem \ref{thm:en2}, Theorem \ref{thm:adde5} and Corollary \ref{coro:adde2}, respectively. To be specific, for a vertex $q\in (0,1)$ of the S-tree, given by the continued fraction $q=[0,a_1,\cdots,a_k]$, it occupies the position of the $N$-th vertex in level $M$, where $M=-1+\Sigma_{i=1}^k a_i$, and $N=1+\Sigma_{i=1}^{k-1}2^{(-1+\Sigma_{j=1}^{i}a_j)}$ ( when $k=1$, set $N=1$). For a rational number $t\in (0,1)$ represented as $t=[0,a_1,a_2,\cdots,a_k]$ where $a_i\ge 1,a_k\ge 2$, its 0-1 sequence of the SC-tree is $$\overline{1\underbrace{0\cdots0}_{a_1-1}1\underbrace{0\cdots0}_{a_{2}-1}1\underbrace{0\cdots0}_{a_{k-1}-1}\cdots 1\underbrace{0\cdots0}_{a_{k}-2}}.$$ As for a rational number $t\in (1,+\infty)$ represented as  $t=[a_1,a_2,\cdots,a_k]$ where $a_i\ge 1,a_k\ge 2$,its 0-1 sequence of the SC-tree is $$\overline{0\underbrace{0\cdots0}_{a_1-1}1\underbrace{0\cdots0}_{a_{2}-1}1\underbrace{0\cdots0}_{a_{k-1}-1}\cdots 1\underbrace{0\cdots0}_{a_{k}-2}}.$$

Intriguing insights emerge regarding the relationships between the SC-tree and the Stern-Brocot tree. These two trees exhibit only sequential differences in the vertices within each level, and we can directly characterize algorithms corresponding to these vertex positions based on their structures. As presented in Theorem \ref{thm:en3}, if the value of  $\overline{1E}$ is $\frac{\Sigma_{i=1}^{t}2^{r_i-1}(2^{k_i+1}-2)}{2}$ and $\overline{1F}$ is $\lfloor\frac{\Sigma_{i=1}^{t}2^{r_i-1}(2^{k_i}+1)}{2}\rfloor$ , where $t,r_i,k_i\in Z_{+}$ and $r_{i+1}-(r_i+k_i)\geq 0$. Then the vertex ordered by $\overline{E}$ of the Stern-Brocot tree is the same as the vertex ordered by $\overline{F}$ of the SC-tree. The same applies to the S-tree and specific segments of the Calkin-Wilf tree.  As shown by Theorem \ref{thm:add2-en}, let the binary expansion of $N_s-1$ be  $\Sigma_{i=1}^{k}2^{r_i}$ ( for $u<v$, $r_u>r_v$ ), then the $N_s$-th vertex in the $M$-th level ($M\ge 1, N_s\le 2^{M-1}$) of the S-tree corresponds the $N_c$-th vertex in the $M$-th level of the Calkin-Wilf tree, where $N_c=\left\{
\begin{aligned}
    &1+\Sigma_{j=1}^{\frac{k}{2}}(2^{r_{2j-1}+1}-2^{r_{2j}+1}),\ \ 2\mid {k} \\
    &1+2^{M}-2^{r_1}+\Sigma_{j=1}^{\frac{k-1}{2}}(2^{r_{2j}+1}-2^{r_{2j+1}+1}),\ \ \ \ \ 2\nmid {k}\\
\end{aligned}
    \right.$. Significantly, these results establish a tight association among the four trees mentioned above, creating a complete linking loop, as illustrated in Figure \ref{image:loop2}. This linkage offers a clearer, more direct, and easier-to-understand relationship between the Stern-Brocot tree and the Calkin-Wilf tree, as all locating algorithms involved rely solely on binary operations. 

As for applications, we reveal that the far-right branch of the S-tree consists of the Fibonacci sequence, as indicated in Corollary \ref{coro:02-en}. Additionally, Corollary \ref{coro:03-en} demonstrates that the problem 2nd of USAMO in 1995\cite{GP} can also be addressed through the structure of the S-tree. 

However, other fascinating applications of the S-tree and the SC-tree remain unexplored and await future research endeavors, including their potential uses in describing the properties of complex numbers $a+bi$, where $a$ and $b$ are rational numbers.

 \begin{figure}[ht]
  \centering
  \includegraphics[width=0.38\textwidth]{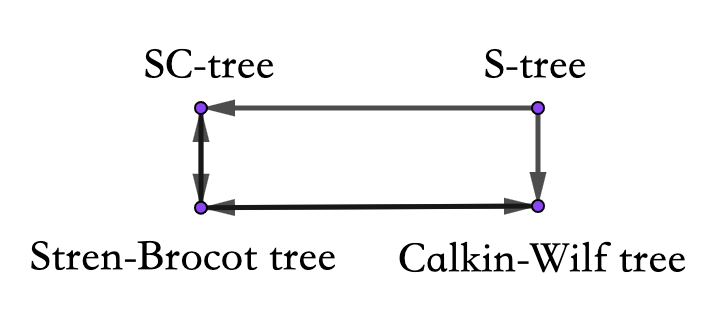}
  \caption{\label{image:loop2}The linking loop among the four trees}
\end{figure} 
\appendix


\begin{acknowledgment}{Acknowledgment.}
We thank Hongjie Zhou for the valuable discussions and for indicating us the problem. 
\end{acknowledgment}

\begin{biog}
\item[Ziting Wang] is a Ph.D. student at Capital Normal University, conducting research in Mathematical Education.
\begin{affil}
Department of Mathematical Sciences, Capital Normal University, Bejing, 100048\\
wangziting@cnu.edu.cn
\end{affil}

\item[Ruijia Guo] is a master's graduate from Capital Normal University, currently actively engaged in teaching activities in a secondary school. 
\begin{affil}
Department of Mathematical Sciences, Capital Normal University, Bejing, 100048\\
\end{affil}

\item[Yixin Zhu] ris a professor and doctoral supervisor at CNU. His expertise lies in the fields of Algebra and Mathematics Education.
\begin{affil}
Department of Mathematical Sciences, Capital Normal University, Bejing, 100048\\
\end{affil}
\end{biog}
\vfill\eject

\end{document}